\documentclass[12pt,reqno]{amsart}
\usepackage{latexsym}
\usepackage{amssymb,amsfonts}
\usepackage{amscd}
\usepackage[nohug]{diagrams}
\numberwithin{equation}{section}

\newtheorem{theorem}{Theorem}
\newtheorem{proposition}[theorem]{Proposition}
\newtheorem{lemma}[theorem]{Lemma}
\newtheorem{corollary}[theorem]{Corollary}

\theoremstyle{remark}
\newtheorem{example}{Example}

\def\Sp{\operatorname{Sp}}
\def\id{\operatorname{id}}

\def\Obj{\operatorname{Ob}}
\def\GF{\operatorname{GF}}

\def\ff{{\boldsymbol f}}
\def\gg{{\boldsymbol g}}
\def\ww{{\boldsymbol w}}
\def\II{{\boldsymbol I}}
\def\JJ{{\boldsymbol J}}
\def\O{\mathbf{\Omega}}
\def\ph{\varphi}
\def\ps{\psi}
\def\Eta{{\boldsymbol\eta}}

\newcommand{\N}{\mathbb{N}}

\newcommand{\dbin}[2]{\left(\kern-0.4em{\binom#1#2}
\kern-0.4em\right)}
\begin{document}
\title[Decomposable functors and the exponential principle, II]
{Decomposable functors and the exponential principle, II}
\newbox\Aut
\setbox\Aut\vbox{
\centerline{\sc Peter J. Cameron$^\dagger$,
Christian Krattenthaler$^{\ddagger}$, \rm and \sc Thomas W. M\"uller$^\dagger$}
\vskip18pt
\centerline{$^\dagger$ School of Mathematical Sciences,}
\centerline{Queen Mary \& Westfield College, University of London,}
\centerline{Mile End Road, London E1 4NS, United Kingdom.}
\centerline{\footnotesize WWW: \tt http://www.maths.qmw.ac.uk/\lower0.5ex\hbox{\~{}}pjc/}
\centerline{\footnotesize WWW: \tt http://www.maths.qmw.ac.uk/\lower0.5ex\hbox{\~{}}twm/}
\vskip18pt
\centerline{$^\ddagger$ Fakult\"at f\"ur Mathematik, Universit\"at Wien,}
\centerline{Nordbergstra\ss e 15, A-1090 Vienna, Austria.}
\centerline{\footnotesize WWW: \footnotesize\tt
http://www.mat.univie.ac.at/\lower0.5ex\hbox{\~{}}kratt} 
}
\author[P. J. Cameron, C. Krattenthaler and T. W. M\"uller]{\box\Aut}

\address{School of Mathematical Sciences, Queen Mary
\& Westfield College, University of London,
Mile End Road, London E1 4NS, United Kingdom.\newline
WWW: {\tt http://www.maths.qmw.ac.uk/\lower0.5ex\hbox{\~{}}pjc/},
\tt http://www.maths.qmw.ac.uk/\lower0.5ex\hbox{\~{}}twm/.}

\address{Fakult\"at f\"ur Mathematik, Universit\"at Wien,
Nordbergstra{\ss}e~15, A-1090 Vienna, Austria.
WWW: \tt http://www.mat.univie.ac.at/\lower0.5ex\hbox{\~{}}kratt.}

\keywords{labelled combinatorial structures, multisort species, 
exponential principle, functor decomposition, magic squares}

\subjclass[2000]{Primary 05A15;
Secondary 05A16 05A19 05C30}

\dedicatory{Dedi\'e \`a la m\'emoire de Pierre Leroux}

\thanks{$^\ddagger$Research partially supported by the Austrian
Science Foundation FWF, grants Z130-N13 and S9607-N13,
the latter in the framework of the National Research Network
``Analytic Combinatorics and Probabilistic Number Theory"}

\begin{abstract}
We develop a new setting for the exponential principle in the context
of multisort
species, where indecomposable objects are generated intrinsically
instead of being given in advance. Our approach uses the language of
functors and natural transformations (composition operators), 
and we show that, somewhat surprisingly, 
a single axiom for the composition already suffices to
guarantee validity of the exponential formula.
We provide various illustrations of our theory, among which are
applications to the enumeration of (semi-)magic squares.
\end{abstract}

\maketitle

\thispagestyle{myheadings}
\font\rms=cmr8 
\font\its=cmti8 
\font\bfs=cmbx8

\markright{\its S\'eminaire Lotharingien de
Combinatoire \bfs 61A \rms (2010), Article~B61Am\hfill}
\def\thepage{}

\section{Introduction}
One of the corner stones of combinatorial enumeration is 
a theory which runs under several
different names, for instance 
{\it theory of species} \cite{BLL}, {\it theory of
exponential structures} \cite[Ch.~5]{StanBI}, 
{\it theory of exponential families} \cite{Wilf}, {\it symbolic method} 
\cite[Part~A]{FlSeAA}, {\it th\'eorie du compos\'e partitionnel} \cite{FoatAZ},
{\it theory of prefabs} \cite{BG},
all of which are more or less equivalent. It is probably fair to say
that the most elaborate of these theories is the theory of species, as
formulated by Joyal \cite{Joyal}
(with the functorial concept of species of structures going back to
Ehresmann \cite{EhreAA}) and further developed by many other
authors. It provides the most general framework for such a theory, at
the expense of employing a rather abstract language, namely that
of category theory. 

A fundamental theorem in each of these theories is the so-called
{\it exponential formula}. Roughly speaking, given a family $\mathcal G$ 
of labelled combinatorial objects (``components"), one produces a
larger family $\mathcal F$ (``composite objects") 
whose objects are obtained by putting together various
elements of $\mathcal G$. The theorem then states that
the (exponential) generating function for $\mathcal F$ equals the
exponential of the (exponential) generating function for $\mathcal G$.

The aim of the present article is to develop a setting, where one
starts with a family $\mathcal F$ of labelled combinatorial objects 
(``composite objects") and a
composition of such objects, and then identifies, in an intrinsic way, a
subfamily $\mathcal G$ of indecomposable objects (``components"), such
that each element of $\mathcal F$ can be decomposed into
objects from $\mathcal G$, and such that the exponential formula holds
for $\mathcal F$ and $\mathcal G$. The main point here is that, in
contrast to the usual set-up for the exponential formula,
indecomposable objects are {\it not\/} given in advance, but are
defined {\it inherently} via the composition operation. 
In particular, our theory leads to a uniform definition of the
property to be {\it indecomposable} for arbitrary labelled combinatorial
objects equipped with a composition operator. Interestingly,
we show that a single
axiom for the composition operator suffices to guarantee validity of the
exponential formula.
The natural language for formulating
a corresponding theory is that of functors and natural
transformations. Consequently, our presentation will be in the context
of species theory.

For ``ordinary" species ($1$-sort species), such a theory has been
presented in \cite{DM} on the basis of two axioms for the composition
operator. In the present article, we extend this
approach to {\it weighted multisort species}. 
Moreover, we show that, actually, one
of the axioms in \cite{DM} can be derived from the other, and that
also in our multivariate setting a single axiom suffices.
Strictly speaking, our presentation does not cover weighted species
(in the sense of \cite[Sec.~6]{Joyal}, \cite[p.~104]{BLL}) 
in full generality; rather, we restrict ourselves to the case where
the defining functor maps to a category of finite sets, thus avoiding
unnecessary technicalities. However, 
extension to the general case of weighted multisort species is
completely straightforward, and is left to the interested
reader (see also Footnote~\ref{foot1}).

\pagenumbering{arabic}
\addtocounter{page}{1}
\markboth{\SMALL P. J. CAMERON, C. KRATTENTHALER AND T. W. M\"ULLER}{\SMALL 
DECOMPOSABLE FUNCTORS AND THE EXPONENTIAL PRINCIPLE, II}

An exponential principle in a wider context, that encompasses
multisort species as a particular case, has been defined by Menni in
\cite{MennAA} (see \cite{MennAB} for further work in this
direction). Indeed, Menni's work and ours partially overlap.
In order to explain the relation between the two, recall that --- as
already pointed out by Joyal \cite[Sec.~7.1]{Joyal} --- multisort
species have the structure of a {\it symmetric monoidal category}.
Now, in the focus of \cite{MennAA} there are {\it simple commutative monoids} in
a given symmetric monoidal category. Menni defines an exponential
principle in this set-up, and he proves this principle to hold for a
large family of symmetric monoidal categories (see
\cite[Prop.~1.4]{MennAA}). 
He shows that this provides a uniform framework for
the exponential principle for numerous variations of species that
had appeared earlier in the literature, including multisort
species (see \cite[Ex.~3.2 and Ex.~3.5 with 
$I=1+1+\dots+1$]{MennAA}).\footnote{Strictly speaking, weights are not discussed
in \cite{MennAA}. However, it would not be difficult to include them
in the theory developed in \cite{MennAA}.} 
However, as it turns out, in the case
of multisort species, Menni's theory does not cover the setting of our
paper. It does apply whenever the considered multisort species,
together with the product induced by our composition operator, forms a
simple commutative monoid. This does not need to be the case, as 
Example~\ref{ex:3} in Section~\ref{sec:illust} shows (see also
Section~\ref{sec:commeta} for more detailed elaboration on these
matters). So, one could say that Menni's theory exhibits the
structural essentials of the exponential principle in a wide
categorical framework, whereas our paper presents a ``minimalistic"
axiomatic setting for the exponential principle that is specific for
multisort species but, as a bonus, includes a wider set of examples 
than Menni's theory does (in the case of multisort species). It is
conceivable that our setting can be adapted to work for some other
kinds of species, but it is unlikely that it can be lifted to the level of
generality of Menni's theory.


In the next section, we develop the general set-up for our theory.
It is formulated within the theory of multisort species, for which we
define certain composition operators $\Eta$ that are subject to
a single axiom, which, in order to be consistent with \cite{DM}, 
we call (D1). 
Furthermore, in the same section,
we present our main results. These are two exponential formulae, 
see Theorems~\ref{Thm:MainThm1} and \ref{Thm:MainThm2}. 
Theorem~\ref{Thm:MainThm2} refines Theorem~\ref{Thm:MainThm1} by
introducing another variable whose powers keep track of the number of
``components." The proof of Theorem~\ref{Thm:MainThm2} 
requires two general facts about our composition operators $\Eta$,
which are presented in Propositions~\ref{lem:F->Fk} and
\ref{prop:Disjoint}. (It is the latter, which, in the less general
context of \cite{DM}, had been assumed as a separate axiom, (D2). As our
proof of Proposition~\ref{prop:Disjoint} shows, 
this was actually not necessary since, within the general
framework, (D2) follows from (D1). It is interesting to note that,
in the context of \cite{MennAA}, Menni also observed that the axiom
(D2) was not necessary; see \cite[Ex.~3.5]{MennAA}. 
Our derivation of (D2) from (D1) provides the
reason why this is the case: we show that (D2) is implied by --- as we call it
--- ``$m$-permutability" of a species together with a composition
operator; see the proof of Proposition~\ref{prop:Disjoint}, and
Lemmas~\ref{lem:assoc} and \ref{lem:Feta-assoc}.
This ``$m$-permutability" comes for free in the context of
\cite{MennAA} since the underlying species is assumed to form a
commutative monoid in the category of species and, thus, satisfies stronger forms
of ``permutatility;" see also the more detailed explanations at the
beginning of Section~\ref{sec:commeta}, and in particular
the paragraph containing \eqref{eq:passoc} and \eqref{eq:pcomm}.) 
The proofs of
Theorems~\ref{Thm:MainThm1} and \ref{Thm:MainThm2}, and of
Propositions~\ref{lem:F->Fk} and \ref{prop:Disjoint}, are given in
Sections~\ref{sec:main1} and \ref{sec:main2}, respectively. They
require a number of auxiliary results, which are established in
Sections~\ref{sec:aux1} and \ref{sec:aux2}, respectively.

Sections~\ref{sec:illust} and \ref{Sec:CMT1} offer illustrations for the
theory developed in Sections~\ref{Sec:MainThm}--\ref{sec:main2}.
Section~\ref{sec:illust} presents three simple examples highlighting
different aspects of combinatorial situations covered by
Theorems~\ref{Thm:MainThm1} and \ref{Thm:MainThm2}. In
Section~\ref{Sec:CMT1}, we show how to apply our results to obtain
generating function identities for (semi-)magic squares, thereby
generalising previous results in the literature. 

The final section, Section~\ref{sec:commeta}, discusses the
relationship between Menni's theory in \cite{MennAA} and ours.
Moreover, there we make Menni's characterisation (see 
\cite[Sec.~2.6]{MennAA}) of the simple commutative monoids 
that are covered by his result \cite[Prop.~1.4]{MennAA} 
on the (general) exponential principle 
explicit for the special case of multisort species.
As we show, this provides, in the language of our paper, a
characterisation of composition operators that are pointwise
associative and commutative (see \eqref{eq:passoc} and
\eqref{eq:pcomm}). Menni's characterisation implies that
a family $\mathcal F$ of labelled combinatorial objects equipped with
such a composition operator can be
re-constructed, in a sense made precise in Theorem~\ref{thm:commeta},
from the standard operation of forming the disjoint 
union in $E(\mathcal G)$
(the species of sets of objects from $\mathcal G$), 
where $\mathcal G$ denotes again the family
of indecomposable objects in $\mathcal F$. We conclude our
paper by ``twisting" this construction (see Theorem~\ref{thm:Etwist}), 
thereby obtaining a large family of examples that fit under our theory 
but not under Menni's.


\section{Set-up and main results}
\label{Sec:MainThm} 

Denote by $\widehat{\bf Set}$ the
category of finite sets and injective mappings, and by
{\bf Set} the subcategory consisting of finite sets and
bijective maps. Moreover, for a positive integer $r$,
let $\mathfrak{D}_r$ be the full subcategory of
${\bf Set}^r\times{\bf Set}^r$ whose objects are given by
\[
\Obj(\mathfrak{D}_r) = \Big\{(\O_1,\O_2)\in \Obj({\bf Set}^r
\times{\bf Set}^r):\,\, \O_1 \cap \O_2  =
\pmb{\emptyset}\Big\},
\]
where $\pmb{\emptyset}=(\emptyset,\ldots,\emptyset)$
is the element of $\Obj({\bf Set}^r)$ all of whose components are
empty, and intersection is componentwise. 

The ingredients needed for our theory are 
{\it $r$-sort species} and certain
{\it composition operators} defined on them.
Recall from \cite{Joyal} (or see \cite[Def.~4 on p.~102]{BLL} for a
definition avoiding the language of category theory) 
that, for a positive integer $r$, an $r$-sort species
is a covariant functor $F: {\bf Set}^r\rightarrow{\bf Set}$.
Given $r$ and an $r$-sort species $F$, the composition operators we
have in mind are certain natural transformations $\Eta$ from the
functor\footnote{The introduction of the category $\mathfrak{D}_r$
corrects a slight imprecision in the set-up of \cite{DM}.}
\[
F\times F:\, \mathfrak{D}_r \overset{(F,F)}
{\longrightarrow} {\bf Set} \times {\bf Set}
\overset{\times}{\longrightarrow} {\bf Set}
\overset{\iota}{\longrightarrow} \widehat{\bf Set}
\]
to the functor
\[
F\circ \amalg:\, \mathfrak{D}_r \overset{\amalg}
{\longrightarrow}{\bf Set}^r \overset{F}{\longrightarrow}
{\bf Set} \overset{\iota}{\longrightarrow}
 \widehat{\bf Set};
\]
that is, families $\Eta=(\eta_{(\O_1,\O_2)})_{(\O_1,\O_2)
\in\Obj(\mathfrak{D}_r)}$ of injective maps,
\begin{equation} \label{eq:Eta}
\eta_{(\O_1,\O_2)}: F[\O_1] \times F[\O_2] \hookrightarrow
F[\O_1 \amalg \O_2],
\end{equation}
such that, for every morphism
$\ff:(\O_1,\O_2)\longrightarrow
(\widetilde{\O}_1,\widetilde{\O}_2)$ of $\mathfrak{D}_r$,
the diagram
\begin{equation}
\label{Eq:NatDiag}
\begin{CD}
F[\O_1]\times F[\O_2] @>\eta_{(\O_1,\O_2)}>>F[\O_1\amalg
\O_2]\\
@V{(F\times F)[\ff]}VV
@VV{(F\circ\amalg)[\ff]}V\\
F[\widetilde{\O}_1]\times F[\widetilde{\O}_2]
@>>{\eta_{(\widetilde{\O}_1,\widetilde{\O}_2)}}>
F[\widetilde{\O}_1 \amalg \widetilde{\O}_2]
\end{CD}
\end{equation}
commutes. Here, $\times$ is the natural product (Cartesian
product) in the category of sets, $\amalg$ is the natural
coproduct (componentwise disjoint union) in the category
${\bf Set}^r$ and in the category ${\bf Set}$ (relying
on the context to clarify the intended meaning), and
$\iota: {\bf Set}\rightarrow\widehat{\bf Set}$ is the
inclusion functor. In what follows, the set-theoretic
operations $\cap, \cup, -$ as well as the inclusion
relation $\subseteq$ and $\vert$ (restriction of morphisms)
in ${\bf Set}^r$ are all understood to be 
componentwise.\footnote{Throughout this paper, we use the symbol $-$ to denote
the difference of sets.}
We shall most of the time drop the indices of $\Eta$-maps when they are
clear from the context, thus writing $\eta(F[\O_1]\times F[\O_2])$
instead of 
$\eta_{(\O_1,\O_2)}(F[\O_1]\times F[\O_2])$, for example.
We shall think of the elements of a set $\eta(F[\O_1]\times F[\O_2])$
as {\it composite objects} within $F[\O_1\amalg\O_2]$.

Given an $r$-sort species $F$ and a composition operator $\Eta$ as above, 
the next step is to identify the subset $F_\Eta[\O]$ of 
{\it``indecomposable" elements} of a set $F[\O]$. 
It is most natural to define
$F_\Eta: \Obj({\bf Set}^r) \rightarrow \Obj({\bf Set})$ 
via
\[
F_\Eta[\O]:= \begin{cases}
                F[\O]\,-\,\underset{\II \neq
\pmb{\emptyset} \neq \JJ}
                {\underset{\II\amalg
                \JJ=\O}
                {\bigcup\limits_{(\II,
                \JJ)\in
                \Obj(\mathfrak{D}_r)}}}\hspace{.8mm}
                \eta\big(F[\II]\times
F[\JJ]\big),&\O\neq\pmb{\emptyset}\\[2mm]
                \emptyset,&\O=\pmb{\emptyset}
             \end{cases}, \qquad \O\in\Obj({\bf Set}^r).
\]
At this point, $F_\Eta$ is just defined as a {\it map} from 
$\Obj({\bf Set}^r)$ to $\Obj({\bf Set})$. In Lemma~\ref{lem:funct} 
in Section~\ref{sec:aux1} we
shall show that $F_\Eta$ is in fact a functor, that is, an $r$-sort species.

Not every natural transformation $\Eta$
is suited for giving rise to an exponential principle. 
We present the single axiom which is needed for this purpose next.
Given $F$, we call a natural transformation
$\Eta: F\times F\rightarrow F\circ\amalg$ a
\textit{composition operator} of $F$, if Axiom~(D1) below holds.

\medskip
(D1) {\it For each $\O\in\Obj({\bf Set}^r)$ and any two
partitions
$(\O_1,\O_2),(\widetilde{\O}_1,
\widetilde{\O}_2)\in\Obj(\mathfrak{D}_r)$
of $\O$ into disjoint parts,
\[
\O_1 \amalg \O_2 = \O =
\widetilde{\O}_1 \amalg \widetilde{\O}_2,
\]
we have that}
\begin{equation}
\label{Eq:DecompAx} 
\eta\big(F[\O_1] \times F[\O_2]\big)
\cap \eta\big(F[\widetilde{\O}_1] \times
F[\widetilde{\O}_2]\big) = \eta\big(\eta\big(F[\O_{11}]
\times F[\O_{12}]\big) \times \eta\big(F[\O_{21}]
\times F[\O_{22}]\big)\big),
\end{equation}
\textit{where $\O_{ij}:= \O_i \cap \widetilde{\O}_j$ for
$i,j\in\{1,2\}$.}

\medskip
An $r$-sort species $F$ will be
called \textit{decomposable}, if $F\neq\emptyset$
(that is, $F[\O]\neq\emptyset$ for some
$\O\in\Obj({\bf Set}^r)$, and if $F$ admits some
composition operator $\Eta$. 

Next, we define {\it weights} on $(F,\Eta)$. 
Fix a commutative ring $\Lambda$ which contains the rational numbers.
A family $\ww=(w_{\O})_{\O\in\Obj({\bf Set}^r)}$ of maps
$w_{\O}: F[\O]\rightarrow \Lambda$
is termed a \textit{$\Lambda$-weight} on $(F,\Eta)$, if
the following three conditions hold:

\medskip
(W0) {\it For all $x\in F[\pmb\emptyset]$, we have
$w_{\pmb{\emptyset}}(x)=1$.}

\smallskip
(W1) \textit{For each morphism
$\ff: \O_1\rightarrow \O_2$ of ${\bf Set}^r,$
the diagram}
\begin{equation*}
\begin{CD}
F[\O_1] @>w_{\O_1}>> \Lambda\\
@V{F[\ff]}VV @VV{\mathrm{id}_\Lambda}V\\
F[\O_2] @>>w_{\O_2}> \Lambda
\end{CD}
\end{equation*}
\textit{commutes.}

\smallskip
(W2) \textit{For each
pair $(\O_1,\O_2)\in\Obj(\mathfrak{D}_r),$ the diagram}
\begin{equation*}
\begin{CD}
F_\eta[\O_1]\times F[\O_2]
@>\eta_{(\O_1,\O_2)}\vert_{F_\Eta[\O_1]\times F[\O_2]}>>
F[\O_1\amalg\O_2]\\
@V{w_{\O_1}\vert_{F_\Eta[\O_1]}\times w_{\O_2}}VV
@VV{w_{\O_1\amalg\O_2}}V\\
\Lambda\times\Lambda @>>\mbox{\scriptsize multiplication in
$\Lambda$}>\Lambda
\end{CD}
\end{equation*}
\textit{commutes.}

\medskip
Here, (W0) and (W1) make $F$ a {\it weighted\/} $r$-sort species
(cf.\ \cite[p.~104]{BLL}),
whereas (W2) demands (in a weak form) 
the $\Lambda$-weight $\ww$ to be compatible with the
composition operator $\Eta$.%
\footnote{\label{foot1}%
As already remarked in the introduction, it would be easy to
generalise our set-up to cover weighted multisort species in full
generality, by relaxing the condition that $F[\O]$ needs to be finite,
and requiring instead that each preimage $w^{-1}_\O(\lambda)$ is
finite and that the ring $\Lambda$ is {\it multiplicatively finite},
in the sense that the number of different 
product representations of $\lambda\in \Lambda$ is always finite.}
In Section~\ref{sec:commeta}, we shall also need the concept of a {\it weak
$\Lambda$-weight}, by which we mean a collection 
$\ww=(w_\O)_{\O\in\Obj(\mathbf{Set}^r)}$ of mappings as above 
satisfying (W0) and (W1), but not necessarily (W2).

Given a $\Lambda$-weight 
$\ww=(w_{\O})_{\O\in\Obj(\mathbf{Set}^r)}$
on $(F,\Eta)$, we define the corresponding {\it 
exponential generating functions}
for $F$ and $F_\Eta$, respectively, by\footnote{For a non-negative integer
$n$, we write $[n]$ for
the standard set $\{1,2,\ldots,n\}$ of cardinality $n$.}
\begin{align*}
\GF_F(z_1,\ldots,z_r) &:=
\sum_{n_1,\ldots,n_r\geq0}\ \sum_{x\in F[([n_1],\dots,[n_r])]}
w_{([n_1],\dots,[n_r])}(x) 
\frac { z_1^{n_1} \cdots
z_r^{n_r}} {n_1!\cdots n_r!},\\
\GF_{F_\Eta}(z_1,\ldots,z_r) &:=
\sum_{n_1,\ldots,n_r\geq0}\ \sum_{x\in F_\Eta[([n_1],\dots,[n_r])]}
w_{([n_1],\dots,[n_r])}(x)
\frac {z_1^{n_1} \cdots
z_r^{n_r}} {n_1!\cdots n_r!},
\end{align*}
where we suppress the dependence on $\ww$ in the notation for better
readability. 

We are now ready to state our first main result, an
\textit{exponential principle}, which generalises 
Part~(a) of the main result in \cite{DM}.

\begin{theorem} \label{Thm:MainThm1} 
Let $r$ be a positive integer,
$F: {\bf Set}^r \rightarrow {\bf Set}$ an $r$-sort species,
and let $\Eta: F\times F\rightarrow F\circ\amalg$ be a
natural transformation. If $F$ is decomposable and $\Eta$
is a composition operator for $F,$ then the generating functions
$\GF_F$ and $\GF_{F_\Eta}$ are connected via the
relation
\begin{equation} \label{Eq:MainThm1} 
\GF_F(z_1,\ldots,z_r) =
\exp\big(\GF_{F_\Eta}(z_1,\ldots,z_r)\big).
\end{equation}
\end{theorem}

The proof of Theorem~\ref{Thm:MainThm1} is given in Section~\ref{sec:main1}.
It requires several preparatory results, which are established
in the next section.

\medskip
In analogy to \cite{DM}, there is a refinement of
Theorem~\ref{Thm:MainThm1} in the spirit of \cite{LaLeAZ,ScSoAA}, 
which we explain next.
Making use of the map $F_\Eta$ defined above, we define a sequence 
of mappings
\[
F_\Eta^{(k)}: \Obj({\bf Set}^r)\rightarrow\Obj({\bf Set}),\quad k\geq0,
\]
with the property that $F_\Eta^{(k)}[\O]\subseteq F[\O]$ by
induction on $k$ via
\[
F_\Eta^{(0)}[\O]:= \begin{cases}
                     F[\pmb{\emptyset}],&\O=\pmb{\emptyset}\\[1mm]
                     \emptyset,&\O\neq\pmb{\emptyset}
                  \end{cases}
\]
and
\begin{equation} \label{eq:Fetakdef}
F_\Eta^{(k)}[\O]:=\underset{\O_1\subseteq\O}{\bigcup_{\O_1\in\Obj({\bf
Set}^r)}}\hspace{.8mm}\eta\big(F_\Eta[\O_1]\times
F_\Eta^{(k-1)}[\O - \O_1]\big),\quad k\geq1.
\end{equation}
As the definition suggests, one should think of $F_\Eta^{(k)}[\O]$ as 
the subset of 
objects in $F[\O]$ consisting of exactly $k$ ``indecomposable" elements
(components).

An immediate induction on
\begin{equation} \label{eq:omabs} 
\|\O\| =
\|(\Omega^{(1)},\Omega^{(2)},\ldots,\Omega^{(r)})\|:=\sum_{j=1} ^{r} 
|\Omega^{(j)}|
\end{equation}
shows that
\begin{equation}
\label{Eq:FetakBoundary} F_\Eta^{(k)}[\O] = \emptyset,\quad
k>\|\O\|.
\end{equation}

Again, by definition, $F_\Eta^{(k)}$ is just a {\it map} from 
$\Obj({\bf Set}^r)$ to $\Obj({\bf Set})$. In Lemma~\ref{lem:kfunct} 
in Section~\ref{sec:aux2} we
shall show that $F_\Eta^{(k)}$ is in fact a functor, that is, an
$r$-sort species. 

It is not difficult to see that, for any $\O\in \Obj(\mathbf{Set}^r)$,
the sets 
$F_\Eta^{(k)}[\O]$ with $k=0,1,2,\dots$ cover all of $F[\O]$.

\begin{proposition} \label{lem:F->Fk}
For every $\O\in\Obj({\bf Set}^r)$, we have
\begin{equation}
\label{Eq:FOmegaFilt} F[\O] = \bigcup_{k\geq0} F_\Eta^{(k)}[\O].
\end{equation}
\end{proposition}

The proof of Proposition~\ref{lem:F->Fk} can be found 
in Section~\ref{sec:main2}.
 
In \cite{DM}, a second axiom, (D2), 
was imposed on the composition operators
$\Eta$ for obtaining a refined exponential principle that takes into
account the filtration given by the sets $F_\Eta^{(k)}[\O]$ appearing
on the right-hand side of \eqref{Eq:FOmegaFilt} (for the case of
$1$-sort species): it required pairwise disjointness of the sets
$F_\Eta^{(k)}[\O]$ for $k=0,1,2,\dots$. 
The proposition below says that,
actually, this assertion is a consequence of Axiom~(D1) (within the general
set-up), and this is even true for multisort species.

\begin{proposition}
\label{prop:Disjoint}
If $F:\mathrm{\bf Set}^r\rightarrow\mathrm{\bf Set}$ is an $r$-sort species
and $\Eta$ is a composition operator for $F,$ then we have, 
for\footnote{We denote by $\N_0$ the set of non-negative
integers.}
$k,\ell\in\N_0$ and $k\neq \ell,$ 
\begin{equation}
\label{Eq:Disjoint}
F_\Eta^{(k)}[\O] \cap F_\Eta^{(\ell)}[\O] =
\emptyset,\quad\O\in\Obj(\mathrm{\bf Set}^r). 
\end{equation}
\end{proposition}

Proposition~\ref{prop:Disjoint} is also proved in Section~\ref{sec:main2}.
Its proof depends crucially on the fact that ``$\Eta$-bracketings"
of $F$-sets and $F_\Eta$-sets do not depend on the order of the terms
$F[\O]$ respectively $F_\Eta[\O]$ involved, nor on the type of 
bracketing used; see Lemmas~\ref{lem:assoc} and \ref{lem:Feta-assoc}
in Sections~\ref{sec:aux1} and \ref{sec:aux2}, respectively. 
From a technical point of view, this is the decisive improvement over
the results in \cite{DM}, and it is the reason that the
dependence of Axiom~(D2) from Axiom~(D1) was not observed there.

Given a $\Lambda$-weight $\ww$ on $(F,\Eta)$, the above 
propositions allow us to refine the weighting to
$$\widetilde w_{\O}(x):=y^kw_{\O}(x),\quad \quad x\in F_\Eta^{(k)}[\O].$$

We then can define the refined generating function
\begin{align*}
\widetilde \GF_{F}(z_1,\ldots,z_r,y)&:= \sum_{n_1,\ldots,n_r\geq0}
\sum_{x\in F[([n_1],\dots,[n_r])]}
\widetilde w_{([n_1],\dots,[n_r])}(x)\,
\frac{z_1^{n_1} \cdots z_r^{n_r} } {n_1!\cdots n_r!}\\
&\hphantom{:}=\sum_{n_1,\ldots,n_r\geq0}\ \sum_{k\geq0} 
\sum_{x\in F_\Eta^{(k)}[([n_1],\dots,[n_r])]}
y^k\,w_{([n_1],\dots,[n_r])}(x)\,
\frac{z_1^{n_1} \cdots z_r^{n_r} } {n_1!\cdots n_r!}.
\end{align*}
With the above notation,
we have the following refinement of Theorem~\ref{Thm:MainThm1},
which is our second main result.

\begin{theorem} \label{Thm:MainThm2}
Under the hypotheses of Theorem~{\em \ref{Thm:MainThm1},} we have
\begin{equation}
\label{Eq:MainThm21}
\widetilde \GF_{F}(z_1,\ldots,z_r,y) = 
\exp\big(y\GF_{F_\Eta}(z_1,\ldots,z_r)\big),
\end{equation}
as well as
\begin{equation}
\label{Eq:MainThm22}
\widetilde \GF_{F}(z_1,\ldots,z_r,y) = 
\big(\GF_F(z_1,\ldots,z_r)\big)^y.
\end{equation}
\end{theorem}

The proof of Theorem~\ref{Thm:MainThm2} 
is given in Section~\ref{sec:main2}, as
a simple consequence of (the proof of) Proposition~\ref{prop:Disjoint}.

\section{Auxiliary results, I}
\label{sec:aux1}

The purpose of this section is to establish several lemmas, which will
be needed in the next section in the proof of
Theorem~\ref{Thm:MainThm1}. At the same time, they also form the basis
for the proofs of the auxiliary results in Section~\ref{sec:aux2},
which eventually will lead to proofs of
Propositions~\ref{lem:F->Fk} and \ref{prop:Disjoint},
and of Theorem~\ref{Thm:MainThm2},
in Section~\ref{sec:main2}. 
In all of this section, we assume that $F$ is a
decomposable $r$-sort species with composition operator $\Eta$.

\begin{lemma} \label{lem:emptyset}
We have $|F[\pmb{\emptyset}]| = 1$.
\end{lemma}
\begin{proof}
By the injectivity of
$\eta_{(\pmb{\emptyset},\pmb{\emptyset})}:
F[\pmb{\emptyset}]\times F[\pmb{\emptyset}]
\rightarrow F[\pmb{\emptyset}]$, the set
$F[\pmb{\emptyset}]$ is either empty or a $1$-set.
Suppose that $F[\pmb{\emptyset}] = \emptyset$.
Choose $\O_1 = (\Omega_1^{(1)},\ldots,
\Omega_1^{(r)}) \in \Obj({\bf Set}^r)$ with
$F[\O_1]\neq\emptyset$, and $\O_2=
(\Omega_2^{(1)}, \ldots, \Omega_2^{(r)})\in\Obj({\bf Set}^r)$
such that $\O_1 \cap \O_2 = \pmb{\emptyset}$ and
$|\Omega_1^{(i)}| = |\Omega_2^{(i)}|$ for $1\leq i\leq r$.
Now consider (D1) for the partition
$\O:=\O_1 \amalg \O_2$, $\widetilde{\O}_i:= \O_i$ for $i=1,2$. By the
functoriality of $F$, we also have $F[\O_2]\neq\emptyset$
and, consequently, the left-hand side of
(\ref{Eq:DecompAx}) is non-empty, whereas the
right-hand side of (\ref{Eq:DecompAx}) would be empty in
case $F[\pmb{\emptyset}]=\emptyset$, a contradiction.
\end{proof}

\begin{lemma}[\sc Commutativity for $(F,\Eta)$] \label{lem:comm}
For every pair
$(\O_1,\O_2)\in\Obj(\mathfrak{D}_r)$,
 we have
\begin{equation}
\label{Eq:EtaComm} \eta\big(F[\O_1] \times F[\O_2]\big) =
\eta\big(F[\O_2] \times F[\O_1]\big).
\end{equation}
\end{lemma}

\begin{proof}
Applying Axiom~(D1) to the partitions
\[
\O:= \O_1 \amalg \O_2 = \O_2 \amalg \O_1,
\]
we find that
\begin{align*}
\mathfrak{I}&:= \eta\big(F[\O_1] \times F[\O_2]\big) \cap
\eta\big(F[\O_2] \times F[\O_1]\big)\\[2mm] &=
\eta\big(\eta\big(F[\O_1 \cap \O_2] \times F[\O_1]\big)
\times \eta\big(F[\O_2] \times F[\O_1 \cap \O_2]\big)\big).
\end{align*}
By Lemma~\ref{lem:emptyset} and injectivity of the $\Eta$-maps, the map
\[
\eta_{(\pmb{\emptyset},\O_1)}: F[\pmb{\emptyset}]
\times F[\O_1] \rightarrow F[\O_1]
\]
is surjective; that is,
\[
\eta\big(F[\O_1 \cap \O_2] \times F[\O_1]\big) = F[\O_1].
\]
Similarly, we have
\[
\eta\big(F[\O_2] \times F[\O_1 \cap \O_2]\big) = F[\O_2].
\]
Thus,
\[
\mathfrak{I} = \eta\big(F[\O_1] \times F[\O_2]\big).
\]
By an analogous application of (D1) and Lemma~\ref{lem:emptyset} to the
partitions
\[
\O = \O_2 \amalg \O_1 = \O_1 \amalg \O_2,
\]
we find that
\[
\mathfrak{I} = \eta\big(F[\O_2] \times F[\O_1]\big),
\]
and the proof is complete.
\end{proof}

\begin{lemma}[\sc $3$-Associativity for $(F,\Eta)$] 
\label{lem:assoc3}
For pairwise
disjoint $\O_1,\O_2,\O_3\in\break\Obj({\bf Set}^r)$, we have
\begin{equation}
\label{Eq:EtaAss3} 
\eta\big(\eta\big(F[\O_1] \times
F[\O_2]\big) \times F[\O_3]\big)= \eta\big(F[\O_1]
\times \eta\big(F[\O_2]\times F[\O_3]\big)\big).
\end{equation}
\end{lemma}

\begin{proof}
We shall show that both sides of (\ref{Eq:EtaAss3}) equal
the intersection
\[
\mathfrak{I}:= \eta\big(F[\O_1 \amalg \O_2] \times
F[\O_3]\big) \cap \eta\big(F[\O_1] \times
F[\O_2 \amalg \O_3]\big).
\]
Applying (D1) to the partitions
\[
\O:= (\O_1 \amalg \O_2) \amalg \O_3 = \O_1 \amalg
(\O_2 \amalg \O_3),
\]
we find that
\[
\mathfrak{I} = \eta\big(\eta\big(F[\O_1] \times
F[\O_2]\big) \times \eta\big(F[\O_1 \cap \O_3]
\times F[\O_3]\big)\big);
\]
and, arguing as in the proof of Lemma~\ref{lem:comm}, this equation
simplifies to
\[
\mathfrak{I} = \eta\big(\eta\big(F[\O_1] \times
F[\O_2]\big)\times F[\O_3]\big).
\]
The same argument, when applied to the partitions
\[
\O = \O_1 \amalg (\O_2 \amalg \O_3) = (\O_1 \amalg \O_2)
\amalg \O_3,
\]
yields
\[
\mathfrak{I} = \eta\big(F[\O_1] \times \eta\big(F[\O_2]
\times F[\O_3]\big)\big),
\]
whence (\ref{Eq:EtaAss3}).
\end{proof}

For the sake of convenience, for pairwise disjoint elements
$\O_1,\dots,\O_m\in \Obj(\mathbf{Set}^r)$,
let us call expressions formed by applying 
$\Eta$-maps to $F[\O_1]$, \dots, $F[\O_m]$ (in any order), 
with each set $F[\O_i]$
occurring exactly once, an {\em $\Eta$-bracketing of
$F[\O_1]$, \dots, $F[\O_m]$}. More formally, for any
$\O\in\Obj(\mathbf{Set}^r)$,
we call the set $F[\O]$ an $\Eta$-bracketing of $F[\O]$; and, given 
an $\Eta$-bracketing $B_1$ of $F[\O_{i_1}]$, \dots, $F[\O_{i_k}]$ and 
an $\Eta$-bracketing $B_2$ of $F[\O_{i_{k+1}}]$, \dots, $F[\O_{i_m}]$, 
with $1\le k\le n-1$ and $\{i_1,i_2,\dots,i_m\}=\{1,2,\dots,m\}$,
the expression
$\eta(B_1\times B_2)$ is, by definition, an $\Eta$-bracketing of
$F[\O_1]$, \dots, $F[\O_m]$. For example, 
the left-hand side and the right-hand side 
of \eqref{Eq:EtaAss3} are two possible
$\Eta$-bracketings of $F[\O_1]$, $F[\O_2]$, $F[\O_3]$, as are
$$\eta\big(\eta\big(F[\O_3] \times
F[\O_2]\big) \times F[\O_1]\big)
\quad \text{and}\quad  
\eta\big(F[\O_3]
\times \eta\big(F[\O_1]\times F[\O_2]\big)\big).
$$

A simple consequence of Lemma~\ref{lem:comm} and (the proof of) 
Lemma~\ref{lem:assoc3} is the following fact.

\begin{corollary} \label{cor:eta3}
All $\Eta$-bracketings of
$F[\O_1]$, $F[\O_2]$, $F[\O_3]$ are equal to
\begin{equation} \label{eq:eta3} 
\eta\big(F[\O_1\amalg \O_2]\times F[\O_3]\big)
\cap\eta\big(F[\O_1\amalg \O_3]\times F[\O_2]\big)
\cap\eta\big(F[\O_2\amalg \O_3]\times F[\O_1]\big).
\end{equation}
\end{corollary}

\begin{proof}
From the proof of Lemma~\ref{lem:assoc3}, we know that
\begin{equation*}
\eta\big(\eta\big(F[\O_1] \times
F[\O_2]\big) \times F[\O_3]\big)= 
 \eta\big(F[\O_1 \amalg \O_2] \times
F[\O_3]\big) \cap \eta\big(F[\O_1] \times
F[\O_2 \amalg \O_3]\big).
\end{equation*}
If, in this equation, we interchange $\O_1$ and $\O_2$ and use
Lemma~\ref{lem:comm} (commutativity), then we obtain
\begin{equation*}
\eta\big(\eta\big(F[\O_1] \times
F[\O_2]\big) \times F[\O_3]\big)= 
 \eta\big(F[\O_1 \amalg \O_2] \times
F[\O_3]\big) \cap \eta\big(F[\O_2] \times
F[\O_1 \amalg \O_3]\big).
\end{equation*}
Both equations together, plus another application of
Lemma~\ref{lem:comm}, imply our claim.
\end{proof}

\begin{lemma}[\sc $4$-Permutability for $(F,\Eta)$] 
\label{lem:assoc4}
If\/ $\O_1,\O_2,\O_3,\O_4$ are pairwise disjoint elements of\/
$\Obj(\mathbf{Set}^r)$, then
all $\Eta$-bracketings of $F[\O_1]$, $F[\O_2]$, $F[\O_3]$, 
$F[\O_4]$ are equal to each other.
\end{lemma}

\begin{proof}
The possible $\Eta$-bracketings of 
$F[\O_1]$, $F[\O_2]$, $F[\O_3]$, $F[\O_4]$ are
\begin{gather} 
\label{eq:4A}
\eta\big(\eta\big(\eta\big(F[\O_1]\times F[\O_2]\big)
\times F[\O_3]\big)\times F[\O_4]\big),\\
\label{eq:4B}
\eta\big(\eta\big(F[\O_1]\times \eta\big(F[\O_2]
\times F[\O_3]\big)\big)\times F[\O_4]\big),\\
\label{eq:4C}
\eta\big(\eta\big(F[\O_1]\times F[\O_2]\big)
\times \eta\big(F[\O_3]\times F[\O_4]\big)\big),\\
\label{eq:4D}
\eta\big(F[\O_1]\times \eta\big(\eta\big(F[\O_2]
\times F[\O_3]\big)\times F[\O_4]\big)\big),\\
\label{eq:4E}
\eta\big(F[\O_1]\times \eta\big(F[\O_2]
\times \eta\big(F[\O_3]\times F[\O_4]\big)\big)\big),
\end{gather}
together with all expressions arising from the above by
permuting $F[\O_1]$, $F[\O_2]$, $F[\O_3]$, $F[\O_4]$.
By Lemma~\ref{lem:assoc3} (3-associativity), 
the bracketing \eqref{eq:4A} equals
the bracketing \eqref{eq:4B}, and the bracketing \eqref{eq:4D} equals
the bracketing \eqref{eq:4E}. It suffices therefore to prove the
equality of \eqref{eq:4A}, \eqref{eq:4C}, \eqref{eq:4E},
and all expressions arising from these three by 
permuting $F[\O_1]$, $F[\O_2]$, $F[\O_3]$, $F[\O_4]$.

Applying (D1) to the partitions
$$\O:=(\O_1\amalg \O_2)\amalg (\O_3\amalg \O_4)=
(\O_1\amalg \O_3)\amalg (\O_2\amalg \O_4),$$
we find that
\begin{multline*}
\eta\big(F[\O_1\amalg\O_2]\times F[\O_3\amalg \O_4]\big)
\cap
\eta\big(F[\O_1\amalg\O_3]\times F[\O_2\amalg \O_4]\big)\\
=
\eta\big(\eta\big(F[\O_1]\times F[\O_2]\big)
\times
\eta\big(F[\O_3]\times F[\O_4]\big)\big).
\end{multline*}
Using this equation, as well as the one which arises by interchanging
$\O_1$ and $\O_2$ and applying
Lemma~\ref{lem:comm} (commutativity) 
on the resulting right-hand side, we obtain
\begin{multline} \label{eq:eta22}
\eta\big(\eta\big(F[\O_1]\times F[\O_2]\big)
\times
\eta\big(F[\O_3]\times F[\O_4]\big)\big)\\
=
\eta\big(F[\O_1\amalg\O_2]\times F[\O_3\amalg \O_4]\big)
\cap
\eta\big(F[\O_1\amalg\O_3]\times F[\O_2\amalg \O_4]\big)\\
\cap
\eta\big(F[\O_2\amalg\O_3]\times F[\O_1\amalg \O_4]\big)
.
\end{multline}
On the other hand, by Corollary~\ref{cor:eta3}, the expression
$$\eta\big(\eta\big(F[\O_1]\times F[\O_2]\big)\times F[\O_3]\big)$$ 
equals the expression \eqref{eq:eta3}. 
Therefore, ``bracketing" these two expressions
by $\eta(\,.\,\times F[\O_4])$, using the injectivity of $\eta$, and 
applying Lemma~\ref{lem:assoc3}
(3-associativity) to the expression resulting from \eqref{eq:eta3},
we arrive at
\begin{multline} \label{eq:eta123}
\eta\big(\eta\big(\eta\big(F[\O_1]\times F[\O_2]\big)\times F[\O_3]\big)
\times F[\O_4]\big)\\
=
\eta\big(F[\O_1\amalg \O_2]\times \eta\big(F[\O_3]\times F[\O_4]\big)\big)
\cap\eta\big(F[\O_1\amalg \O_3]\times \eta\big(F[\O_2]\times
F[\O_4]\big)\big)\\ 
\cap\eta\big(F[\O_2\amalg \O_3]\times \eta\big(F[\O_1]\times F[\O_4]\big)\big).
\end{multline}
Since $\eta\big(F[\O_3]\times F[\O_4]\big)\subseteq
F[\O_3\amalg\O_4]$, and similar inclusions hold for other
combinations of the $\O_i$'s, we have
\begin{align*} 
\eta\big(F[\O_1\amalg \O_2]\times \eta\big(F[\O_3]\times
F[\O_4]\big)\big)
&\subseteq \eta\big(F[\O_1\amalg \O_2]\times F[\O_3\amalg\O_4]\big),
\\
\eta\big(F[\O_1\amalg \O_3]\times \eta\big(F[\O_2]\times
F[\O_4]\big)\big)
&\subseteq \eta\big(F[\O_1\amalg \O_3]\times F[\O_2\amalg\O_4]\big),
\\
\eta\big(F[\O_2\amalg \O_3]\times \eta\big(F[\O_1]\times
F[\O_4]\big)\big)
&\subseteq \eta\big(F[\O_2\amalg \O_3]\times F[\O_1\amalg\O_4]\big).
\end{align*}
Altogether, these inclusions imply that the right-hand side of
\eqref{eq:eta123} is contained in the right-hand side of
\eqref{eq:eta22}. We infer that
\begin{equation} \label{eq:22eta123} 
\eta\big(\eta\big(\eta\big(F[\O_1]\times F[\O_2]\big)\times F[\O_3]\big)
\times F[\O_4]\big)\subseteq
\eta\big(\eta\big(F[\O_1]\times F[\O_2]\big)
\times
\eta\big(F[\O_3]\times F[\O_4]\big)\big).
\end{equation}
However, by injectivity of the $\Eta$-maps, both sides of
\eqref{eq:22eta123}
have cardinality 
$$\vert F[\O_1]\vert\cdot \vert F[\O_2]\vert\cdot 
\vert F[\O_3]\vert\cdot \vert F[\O_4]\vert.$$ 
Hence, they must be
equal, which proves the equality of \eqref{eq:4A} and \eqref{eq:4C}.

An analogous argument proves equality of \eqref{eq:4C} and
\eqref{eq:4E}.

Since, by Lemma~\ref{lem:comm} (commutativity), 
the right-hand side of \eqref{eq:eta22} is invariant under
permutation of $\O_1,\O_2,\O_3,\O_4$, the proof is complete.
\end{proof}

The (proof of the) above lemma, combined with Lemma~\ref{lem:comm} and
Corollary~\ref{cor:eta3}, leads to the following observation.

\begin{corollary} \label{cor:eta4}
All $\Eta$-bracketings of
$F[\O_1]$, $F[\O_2]$, $F[\O_3]$, $F[\O_4]$ are equal to
\begin{multline} \label{eq:eta4} 
\eta\big(F[\O_1\amalg\O_2]\times F[\O_3\amalg \O_4]\big)
\cap
\eta\big(F[\O_1\amalg\O_3]\times F[\O_2\amalg \O_4]\big)\\
\cap
\eta\big(F[\O_2\amalg\O_3]\times F[\O_1\amalg \O_4]\big)\\
\cap\eta\big(F[\O_1] \times F[ \O_2\amalg \O_3 \amalg\O_4]\big)
\cap\eta\big(F[\O_2] \times F[ \O_1\amalg \O_3 \amalg\O_4]\big)\\
\cap\eta\big(F[\O_3] \times F[ \O_1\amalg \O_2 \amalg\O_4]\big)
\cap\eta\big(F[\O_4] \times F[ \O_1\amalg \O_2 \amalg\O_3]\big).
\end{multline}
\end{corollary}

\begin{proof}
By Corollary~\ref{cor:eta3}, we have
\begin{multline*}
\eta\big(F[\O_1\amalg \O_2]\times \eta\big(F[\O_3]\times
F[\O_4]\big)\big)\\
=
\eta\big(F[\O_1\amalg\O_2\amalg \O_3]\times F[\O_4]\big)
\cap\eta\big(F[\O_1\amalg \O_2\amalg \O_4]\times F[\O_3]\big)\\
\cap\eta\big(F[\O_3\amalg \O_4]\times F[\O_1\amalg\O_2]\big).
\end{multline*}
Analogous identities hold for the other two terms on the right-hand
side of \eqref{eq:eta123}. If we combine this with Lemma~\ref{lem:comm}
(commutativity) and Lemma~\ref{lem:assoc4} (4-permutability), 
then the claim follows immediately. 
\end{proof}

Before we state the general permutability result, let us
introduce the following short-hand notation,
which will be used in its proof. Given a subset $I$ of $[m]$,
where $I=\{i_1,i_2,\dots,i_k\}$ with $i_1<i_2<\dots<i_k$,
we let $\O_I$ stand for $\O_{i_1}\amalg\dots\amalg \O_{i_k}$.

\begin{lemma}[\sc $m$-Permutability for $(F,\Eta)$] 
\label{lem:assoc}
If\/ $\O_1,\dots,\O_m$ are pairwise disjoint elements of\/
$\Obj(\mathbf{Set}^r)$, then
all $\Eta$-bracketings of $F[\O_1]$, \dots, $F[\O_m]$ are equal to each other.
\end{lemma}

\begin{proof}
We shall prove by induction on $m$ that, for each $m\ge2$, 
all $\Eta$-bracketings of $F[\O_1]$, \dots, $F[\O_m]$ equal
\begin{equation} \label{eq:mbrack} 
\underset{I\ne \emptyset\ne J}{\underset{I\cup J=[m],\
I\cap J=\emptyset}{\bigcap_{I,J\subseteq [m]}}}
\eta\big(F[\O_I]\times F[\O_J]\big).
\end{equation}

For $m=2$, the assertion follows from Lemma~\ref{lem:comm}
(commutativity). For $m=3$, the assertion is equivalent to
Corollary~\ref{cor:eta3} (again, modulo Lemma~\ref{lem:comm}), 
and for $m=4$, the assertion is equivalent to
Corollary~\ref{cor:eta4}.

Now let $m\ge5$, and let us suppose that the assertion is true up to $m-1$.
Consider an $\Eta$-bracketing of $\O_1,\dots,\O_m$. There are three
possibilities. Either this bracketing has the form
\begin{equation} \label{eq:Form1}
\eta\big(\eta\big(E_1\big) \times\eta\big(E_2\big) \big) ,
\end{equation}
where $E_1$ and $E_2$ are expressions involving $\Eta$-maps and 
distinct $\O_i$'s, each of them involving at least two $\O_i$'s, 
or the form
\begin{equation} \label{eq:Form2}
\eta\big(F[\O_i] \times\eta\big(E_3\big) \big) ,
\end{equation}
where $E_3$ involves $\Eta$-maps and
$\O_1,\dots,\O_{i-1},\O_{i+1},\dots,\O_m$, for some $i$, or the form
\begin{equation} \label{eq:Form3}
\eta\big(\eta\big(E_4\big)\times F[\O_i] \big) ,
\end{equation}
where $E_4$ involves $\Eta$-maps and
$\O_1,\dots,\O_{i-1},\O_{i+1},\dots,\O_m$, for some $i$.

We start by considering \eqref{eq:Form1}. Let us assume that
$E_1$ involves all $\O_r$'s for $r\in R$, and $E_2$ involves all $\O_s$'s for
$s\in S$, with $R\cup S=[m]$, $R\cap S=\emptyset$, $R\ne \emptyset\ne
S$. By the inductive hypothesis, we know that 
\begin{equation} \label{eq:E1} 
\eta(E_1)=\underset{I_1\ne \emptyset\ne J_1}{\underset{I_1\cup J_1=R,\
I_1\cap J_1=\emptyset}{\bigcap_{I_1,J_1\subseteq R}}}
\eta\big(F[\O_{I_1}]\times F[\O_{J_1}]\big)
\end{equation}
and
\begin{equation} \label{eq:E2} 
\eta(E_2)=\underset{I_2\ne \emptyset\ne J_2}{\underset{I_2\cup J_2=S,\
I_2\cap J_2=\emptyset}{\bigcap_{I_2,J_2\subseteq S}}}
\eta\big(F[\O_{I_2}]\times F[\O_{J_2}]\big).
\end{equation}
If we substitute \eqref{eq:E1} and \eqref{eq:E2}
in \eqref{eq:Form1} and use injectivity
of the $\Eta$-maps, then we obtain
$$
\underset{I_1\ne \emptyset\ne J_1}{\underset{I_1\cup J_1=R,\
I_1\cap J_1=\emptyset}{\bigcap_{I_1,J_1\subseteq R}}}
\ \
\underset{I_2\ne \emptyset\ne J_2}{\underset{I_2\cup J_2=S,\
I_2\cap J_2=\emptyset}{\bigcap_{I_2,J_2\subseteq S}}}
\eta\big(\eta\big(F[\O_{I_1}]\times F[\O_{J_1}]\big)
\times
\eta\big(F[\O_{I_2}]\times F[\O_{J_2}]\big)\big).
$$
We may now apply Corollary~\ref{cor:eta4} to each of the terms on the
right-hand side of this equation. It is not difficult to see that, 
together with Lemma~\ref{lem:comm} (commutativity), we obtain
\eqref{eq:mbrack}. 

Next we consider \eqref{eq:Form2}. 
By the inductive hypothesis, we know that 
$$
\eta(E_3)=\underset{I\ne \emptyset\ne J}{\underset{I\cup
J=[m]- \{i\},\
I\cap J=\emptyset}{\bigcap_{I,J\subseteq [m]-\{i\}}}}
\eta\big(F[\O_{I}]\times F[\O_{J}]\big).
$$
If we substitute this
in \eqref{eq:Form2} and use injectivity
of the $\Eta$-maps, then we obtain
$$
\underset{I\ne \emptyset\ne J}{\underset{I\cup J=[m]-\{i\},\
I\cap J=\emptyset}{\bigcap_{I,J\subseteq [m]-\{i\}}}}
\eta\big(F[\O_i]\times 
\eta\big(F[\O_{I}]\times F[\O_{J}]\big)\big).
$$
We may now apply Corollary~\ref{cor:eta3} to each of the terms on the
right-hand side of this equation. It is not difficult to see that, 
together with Lemma~\ref{lem:comm} (commutativity), we obtain
\eqref{eq:mbrack}. 

The argument for \eqref{eq:Form3} is analogous. This completes the
proof of the lemma.
\end{proof}

For 
$\O = (\Omega^{(1)},\ldots,\Omega^{(r)}) \in \Obj({\bf Set}^r)$
and an integer $\rho\in[r]$, we write
$(\omega,\rho)\in\O$ to mean $\omega\in\Omega^{(\rho)}$.
This is the concept of {\it base point\/} needed in the present
context.

\begin{lemma} \label{lem:Fzerl}
For non-empty $\O \in \Obj({\bf Set}^r)$ and
every choice of base point $(\omega,\rho)\in\O,$ we have
\begin{equation}
\label{Eq:FOmegaDecomp} F[\O]=
\underset{(\omega,\rho)\in\O_1\subseteq\O}
{\bigcup_{\O_1\in\Obj({\bf Set}^r)}} \eta\big(F_\Eta[\O_1]
\times F[\O -\O_1]\big).
\end{equation}
\end{lemma}

\begin{proof}
Let $x\in F[\O]$ be an arbitrary element, and consider the
totality of all $\O_1 \in \Obj({\bf Set}^r)$ such that
$(\omega,\rho)\in\O_1\subseteq\O$ and $x\in\eta\big(F[\O_1] \times
F[\O-\O_1]\big)$. Such $\O_1$'s do exist; for instance
$\O_1=\O$ has these properties, since the map
\[
\eta_{(\O,\pmb{\emptyset})}: F[\O] \times
F[\pmb{\emptyset}] \hookrightarrow F[\O]
\]
is surjective. Among these $\O_1$'s we choose one
of minimal norm $\|\O_1\|$ (recall the definition in
\eqref{eq:omabs}), say $\O_1(x)$. 
Now suppose
that $x\not\in\eta\big(F_\Eta[\O_1(x)] \times
F[\O - \O_1(x)]\big)$. Then, by 
the definition of $F_\Eta$, the injectivity of
$\eta$, and the choice of $\O_1(x)$, we must have 
\begin{multline*}
x \in \eta\Big(\big(F[\O_1(x)] - F_\Eta[\O_1(x)]\big)
\times F[\O-\O_1(x)]\Big)\\[2mm]
= \underset{\II\neq\pmb{\emptyset}\neq
\JJ}{\underset{\II\amalg
\JJ=\O_1(x)}{\bigcup_{(\II,
\JJ)\in\Obj(\mathfrak{D}_r)}}}\eta\Big(\eta
\big(F[\II] \times F[\JJ]\big)\times
F[\O-\O_1(x)]\Big).
\end{multline*}
Consequently, there exists $(\II_1,
\JJ_1)\in\Obj(\mathfrak{D}_r)$ such that
$\II_1\amalg\JJ_1=\O_1(x)$,
$\II_1\neq\pmb{\emptyset}\neq
\JJ_1$, and
\[
x \in \eta\Big(\eta\big(F[\II_1] \times
F[\JJ_1]\big) \times F[\O-\O_1(x)]\Big).
\]
Using Corollary~\ref{cor:eta3} (3-permutability),
we see that the latter set is
contained in both $\eta\big(F[\II_1] \times
F[\O-\II_1]\big)$ and $\eta
\big(F[\JJ_1]\times F[\O-\JJ_1]\big)$.
The base point $(\omega,\rho)$ is contained in
$\II_1$ or $\JJ_1$; to fix ideas, say $(\omega,\rho)\in\II_1$. 
Hence, we arrive at
the assertion that
\[
x\in\eta\big(F[\II_1] \times
F[\O-\II_1]\big),\,(\omega,\rho)\in
\II_1\subseteq\O,\,\|\II_1\|<\|\O_1(x)
\|,
\]
contradicting the choice of $\O_1(x)$. We conclude that
$x$ is indeed contained in
\[
\eta\big(F_\Eta[\O_1(x)] \times F[\O - \O_1(x)]\big),
\]
and (\ref{Eq:FOmegaDecomp}) is proven.
\end{proof}

\begin{lemma} \label{lem:Fdisj}
The right-hand side of \eqref{Eq:FOmegaDecomp} is a
disjoint union.
\end{lemma}

\begin{proof}
In the context of Lemma~\ref{lem:Fzerl}, let
$\O_1,\O_2\in\Obj(\mathbf{Set}^r)$ be such that
$(\omega,\rho)\in\O_i\subseteq\O$ and
$\O_1\neq\O_2$, say $\O_1\not\subseteq\O_2$. It is enough
to show that
\[
\mathfrak{I}:= \eta\big(F[\O_1] \times F[\O - \O_1]\big)
\cap \eta\big(F[\O_2] \times F[\O - \O_2]\big)
\]
has an empty intersection with $\eta\big(F_\Eta[\O_1]
\times F[\O - \O_1]\big)$. But, by (D1), we have
\[
\mathfrak{I} \subseteq \eta\Big(\eta\big(F[\O_1\cap\O_2]
\times F[\O_1-\O_2]\big) \times F[\O-\O_1]\Big),
\]
and, by definition of $F_\Eta$ and the fact that
$(\omega,\rho)\in\O_1\cap\O_2 \neq \pmb{\emptyset}
\neq \O_1 - \O_2$, we have
\[
F_\Eta[\O_1] \cap \eta\big(F[\O_1\cap\O_2]\times
F[\O_1-\O_2]\big) = \emptyset.
\]
Consequently, by the injectivity of $\eta$, we must indeed
have
\begin{multline*}
\eta\big(F_\Eta[\O_1] \times F[\O-\O_1]\big) \cap
\mathfrak I\\[2mm]
\subseteq \eta\big(F_\Eta[\O_1]\times F[\O-\O_1]\big)
\cap \eta\Big(\eta\big(F[\O_1\cap \O_2] \times
F[\O_1-\O_2]\big)\times F[\O-\O_1]\Big) = \emptyset,
\end{multline*}
as required.
\end{proof}

\begin{lemma}[\sc Functoriality of $F_\Eta$]
\label{lem:funct} 
Let $\O,\widetilde\O\in\Obj({\bf Set}^r)$, and let
$\ff:\O\rightarrow\widetilde\O$ be a morphism. Then
\[
F[\ff](F_\Eta[\O]) = F_\Eta[\widetilde\O];
\]
that is, setting
$F_\Eta[\ff]:=F[\ff]
\vert_{F_\Eta[\O]},$ we get a functor
$F_\Eta: {\bf Set}^r\rightarrow {\bf Set}$.
\end{lemma}

\begin{proof}
The assertion is obvious if $\O=\pmb{\emptyset}$, so we may
suppose that $\O\neq\pmb{\emptyset}$. Then, using
the naturality of $\Eta$ (that is, the diagram
~(\ref{Eq:NatDiag})), we have
\begin{align*}
F[\ff](F_\Eta[\O]) &= F[\ff]
\Bigg(F[\O]\,-\,\underset{\II_1\neq
\pmb{\emptyset}\neq\JJ_1}
{\underset{\II_1\amalg\JJ_1=\O}
{\bigcup_{(\II_1,\JJ_1)\in
\Obj(\mathfrak{D}_r)}}} \eta\big(F[\II_1] \times
F[\JJ_1]\big)\Bigg)\\[2mm] &=
F[\widetilde\O]\,-\,\underset{\II_1\neq
\pmb{\emptyset}\neq\JJ_1}
{\underset{\II_1\amalg\JJ_1=\O}
{\bigcup_{(\II_1,\JJ_1)\in
\Obj(\mathfrak{D}_r)}}}\big(F\circ\amalg\big)
\big[(\ff\vert_{\II_1},
\ff\vert_{\JJ_1})\big]
\big(\eta\big(F[\II_1] \times
F[\JJ_1]\big)\big)\\[2mm] &=
F[\widetilde\O]\,-\,\underset{\II_1\neq
\pmb{\emptyset}\neq\JJ_1}
{\underset{\II_1 \amalg\JJ_1=\O}
{\bigcup_{(\II_1,\JJ_1)\in
\Obj(\mathfrak{D}_r)}}}\eta\big(F[\ff
\vert_{\II_1}](F[\II_1])\times
F[\ff\vert_{\JJ_1}]
(F[\JJ_1])\big)\\[2mm]
 &= F[\widetilde\O]\,-\,\underset{\II_2
 \neq\pmb{\emptyset}\neq\JJ_2}
{\underset{\II_2 \amalg\JJ_2=\widetilde\O}
{\bigcup_{(\II_2,\JJ_2)\in
\Obj(\mathfrak{D}_r)}}}\eta\big(F[\II_2]
\times F[\JJ_2]\big)\\[2mm]
&= F[\widetilde\O].
\end{align*}
\end{proof}

\section{Proof of Theorem~\ref{Thm:MainThm1}}
\label{sec:main1}

For convenience, let us ``extend" the $\Lambda$-weight $\ww$ to subsets
of $F[\O]$, for all $\O\in\Obj(\mathbf{Set}^r)$.
To be precise, for $A\subseteq F[\O]$, we define
$$
w_\O\big(A\big):=
\sum _{x\in A} ^{}w_{\O}(x).
$$
Given disjoint $\O_1,\O_2\in\Obj(\mathbf{Set}^r)$ and subsets
$A_1\subseteq F_\Eta[\O_1]$ and $A_2\subseteq F[\O_2]$, Axiom~(W2)
says that 
\begin{equation} \label{eq:wmult} 
w_{\O_1\amalg \O_2}(\eta(A_1\times A_2))
=
w_{\O_1}(A_1)\cdot w_{\O_2}(A_2).
\end{equation}
In the sequel, we shall suppress the indices of weights $w$ for better
readability, the indices always being clear from the context.

As a direct consequence of Lemmas~\ref{lem:Fzerl} and \ref{lem:Fdisj}, 
of the injectivity of the $\Eta$-maps, and
of \eqref{eq:wmult}, we have that, for $n_1,n_2,\dots,n_r\in \mathbb N_0$
and $n_\rho>0$,
\begin{multline}
\label{Eq:psiFEval1} 
w\big(F[([n_1],\ldots,
[n_r])]\big)\\ =
\underset{1\in\Omega_1^{(\rho)}}{\sum_{\Omega_1^{(i)}
\subseteq [n_i]\,(1\leq i\leq r)}}\,
w\big(F_\Eta[(\Omega_1^{(1)},\ldots,
\Omega_1^{(r)})]\big)\cdot w\big(F
[([n_1] - \Omega_1^{(1)},\ldots,
[n_r]-\Omega_1^{(r)})]\big).
\end{multline}
Using the functoriality of $F$ and $F_\Eta$, together with
Axiom~(W1), each $\O_1$ with
\[
\O_1 = (\Omega_1^{(1)},\ldots,\Omega_1^{(r)})\subseteq
([n_1],\ldots, [n_r]),
\]
$(1,\rho)\in\O_1$, and cardinalities
$|\Omega_1^{(i)}|=\mu_i$ ($1\leq i\leq r$), is seen to contribute
\begin{equation}
\label{Eq:IntCount}
w\big(F_\Eta[([\mu_1],\ldots,
[\mu_r])]\big)\cdot w\big(F
[([n_1-\mu_1] ,\ldots,
[n_r-\mu_r])]\big)
\end{equation}
to the right-hand side of (\ref{Eq:psiFEval1}). We observe
that (\ref{Eq:IntCount}) does not depend upon
$\O_1$ itself, but only on the cardinalities
$\mu_1,\ldots, \mu_r$ of the components
$\Omega_1^{(1)},\ldots, \Omega_1^{(r)}$. Therefore, the
\[
\frac{\mu_\rho}{n_\rho} \prod_{1\leq i\leq r} \binom{n_i}{\mu_i}
\]
elements $\O_1\in\Obj(\mathbf{Set}^r)$ with 
$\O_1\subseteq ([n_1],\ldots,[n_r])$,
$(1,\rho)\in\O_1$, and such that
$|\Omega_1^{(i)}| = \mu_i$ for $1\leq i\leq r$, contribute
\[
\frac{\mu_\rho}{n_\rho}\, \bigg(\prod_{1\leq i\leq r}
\binom{n_i}{\mu_i}\bigg)
w\big(F_\Eta[([\mu_1],\ldots,
[\mu_r])]\big)\cdot w\big(F
[([n_1-\mu_1] ,\ldots,
[n_r-\mu_r])]\big)
\]
to the right-hand side of (\ref{Eq:psiFEval1}), and we
obtain
\begin{multline*}
w\big(F[([n_1],\ldots, [n_r])]\big) =
{\sum_{0\leq\mu_i\leq n_i\,(1\leq i\leq
r)}}\,
\frac{\mu_\rho}{n_\rho}\, \bigg(\prod_{1\leq i\leq r}
\binom{n_i}{\mu_i}\bigg) 
w\big(F_\Eta[([\mu_1],\ldots,
[\mu_r])]\big)\\
\kern7cm
\cdot w\big(F
[([n_1-\mu_1] ,\ldots,
[n_r-\mu_r])]\big)
,\\
n_\rho>0,
\end{multline*}
or, equivalently, 
\begin{multline}
\label{Eq:psiFEval2} 
n_\rho\cdot w\big(F[([n_1],\ldots, [n_r])]\big) =
{\sum_{0\leq\mu_i\leq n_i\,(1\leq i\leq
r)}}\,
\mu_\rho \bigg(\prod_{1\leq i\leq r}
\binom{n_i}{\mu_i}\bigg) 
w\big(F_\Eta[([\mu_1],\ldots,
[\mu_r])]\big)\\
\kern7cm
\cdot w\big(F
[([n_1-\mu_1] ,\ldots,
[n_r-\mu_r])]\big)
,
\end{multline}
as long as $n_\rho>0$. However, Equation~\eqref{Eq:psiFEval2} holds
as well for $n_\rho=0$, with both sides vanishing, 
so that we are allowed to drop the restriction $n_\rho>0$.

Fix $\rho\in[r]$, multiply both sides of
(\ref{Eq:psiFEval2}) by
\[
 {z_1^{n_1} \cdots z_{\rho-1}^{n_{\rho-1}}
z_\rho^{n_\rho-1}z_{\rho+1}^{n_{\rho+1}}
\cdots z_r^{n_r}} /
{(n_1! \cdots n_r!)},
\]
and sum over all tuples $(n_1,\ldots,n_r)\in\N_0^r$,
to get
\begin{multline}
\label{Eq:PsiFEval1}
{\sum_{n_1,\ldots,n_r\geq0}}\,
w\big(F[([n_1],\ldots,[n_r])]\big)\,\frac { z_1^{n_1} \cdots
 z_{\rho-1}^{n_{\rho-1}}
z_\rho^{n_\rho-1}z_{\rho+1}^{n_{\rho+1}}
\cdots z_r^{n_r}} {n_1! \cdots n_{\rho-1}!\,
(n_\rho-1)!\,n_{\rho+1}!
\cdots n_r!}\\ 
\kern-4cm =
{\sum_{n_1,\ldots,n_r\geq0}}\
{\sum_{0\leq\mu_i\leq n_i\,(1\leq i\leq
r)}}\,
\frac{w\big(F_\Eta[([\mu_1],\ldots,[\mu_r])]\big)}
{\mu_1! \cdots \mu_{\rho-1}!\,
(\mu_\rho-1)!\, \mu_{\rho+1}!\cdots \mu_r!}\\
\cdot\frac{w\big(F[([n_1-\mu_1],
\ldots,[n_r-\mu_r])]\big)}{(n_1-\mu_1)! \cdots
(n_r-\mu_r)!}\,
z_1^{n_1} \cdots  z_{\rho-1}^{n_{\rho-1}}
z_\rho^{n_\rho-1}z_{\rho+1}^{n_{\rho+1}}
\cdots z_r^{n_r},
\end{multline}
where $1/(-1)!$ has to be interpreted as $0$.
The left-hand side equals
\[
\frac{\partial \GF_F}{\partial z_\rho},
\]
while the right-hand side is identified as
\[
\frac{\partial \GF_{F_\Eta}}{\partial z_\rho}
\,\GF_F;
\]
whence the equations
\begin{equation}
\label{Eq:PsiFEtaEval2}
\frac{\partial \GF_F}{\partial
z_\rho} = \frac{\partial \GF_{F_\Eta}}{\partial z_\rho}
\,\GF_F, \qquad 1\leq\rho\leq r.
\end{equation}
Set
\[
Q(z_1,\ldots,z_r):=
\GF_F(z_1,\ldots,z_r)
\exp\big(-\GF_{F_\Eta}(z_1,\ldots,z_r)\big).
\]
Then, in view of Equations (\ref{Eq:PsiFEtaEval2}), the
series $Q$ satisfies
\[
\frac{\partial Q}{\partial z_\rho} = 0,
\qquad 1\leq\rho\leq r.
\]
These last equations force $Q$ to be independent of $z_1,z_2,\dots,z_r$.
However, since\break $\GF_{F_\Eta}(0,\ldots,0) = 0$
by definition of $F_\Eta$, and
since
$\GF_F(0,\ldots,0)=1$
by Axiom~(W0) and Lemma~\ref{lem:emptyset}, direct inspection shows that
\[
Q(0,0,\dots,0) = 1,
\]
and \eqref{Eq:MainThm1} follows.
\qed

\section{Auxiliary results, II}
\label{sec:aux2}

In this section, we complement the results of Section~\ref{sec:aux1}
by establishing several further results which will be 
needed in the proofs of 
Proposition~\ref{prop:Disjoint} and Theorem~\ref{Thm:MainThm2}, 
to be given in the next section.  
The first lemma provides the analogue of Lemma~\ref{lem:assoc} for
$F_\Eta$, namely that arbitrary permutability holds also for
$\Eta$-bracketings of $F_\Eta$-images (see the subsequent paragraph
for the precise definition). All the remaining lemmas concern the maps
$F_\Eta^{(k)}$. 
In all of this section, we assume that $F$ is a
decomposable $r$-sort species with composition operator $\Eta$.

In complete analogy with the corresponding definition in
Section~\ref{sec:aux1}, 
we define the\break concept of an {\em$\Eta$-bracketing of 
$F_\Eta[\O_1], \dots, F_\Eta[\O_m]$} for pairwise disjoint elements\break
$\O_1,\dots,\O_m\in \Obj(\mathbf{Set}^r)$: simply replace $F$ by $F_\Eta$
everywhere in the definition just after the proof of Lemma~\ref{lem:assoc3}.

\begin{lemma}[\sc $m$-Permutability for $(F_\Eta,\Eta)$] 
\label{lem:Feta-assoc}
If\/ $\O_1,\dots,\O_m$ are pairwise disjoint elements of\/
$\Obj(\mathbf{Set}^r)$, then
all $\Eta$-bracketings of $F_\Eta[\O_1]$, \dots, 
$F_\Eta[\O_m]$ are equal to each other.
\end{lemma}

\begin{proof}
Since $F_\Eta[\pmb\emptyset]=\emptyset$ by definition of $F_\Eta$, our
claim holds if at least one of $\O_1,\dots,\O_m$ equals
$\pmb\emptyset$. Hence, we may assume that all of $\O_1,\dots,\O_m$
are non-empty. 

Next we note that, for sets $M_1,\dots,M_m,A_1,\dots,A_m$, we have
\begin{multline} \label{eq:complement}
(M_1- A_1)\times (M_2- A_2)\times\dots\times 
(M_m- A_m)\\=M_1\times M_2\times \cdots \times M_m-
\Bigg(\bigcup_{k=1}^m
M_1\times\cdots\times M_{k-1}\times A_k\times M_{k+1}\times \cdots
\times M_m\Bigg).
\end{multline}
Now assume that we are given two $\Eta$-bracketings of
$F_\Eta[\O_1],\dots,F_\Eta[\O_m]$, say
$$B_\Eta\big(F_\Eta[\O_1],\dots,F_\Eta[\O_m]\big)
\quad \text{and}\quad 
\bar B_\Eta\big(F_\Eta[\O_1],\dots,F_\Eta[\O_m]\big).$$ 
Substituting the definition of $F_\Eta$ into
$B_\Eta\big(F_\Eta[\O_1],\dots,F_\Eta[\O_m]\big)$, and applying
Ident\-ity~\eqref{eq:complement} plus injectivity of $\Eta$-maps, we
find that
\begin{multline*} 
B_\Eta\big(F_\Eta[\O_1],\dots,F_\Eta[\O_m]\big)
=B_\Eta\big(F[\O_1],\dots,F[\O_m]\big)\\
-
\bigcup_{k=1}^m
B_\Eta\Bigg(
F[\O_1],\dots, F[\O_{k-1}], 
\underset{\II \neq
\pmb{\emptyset} \neq \JJ}
                {\underset{\II\amalg
                \JJ=\O_k}
                {\bigcup\limits_{(\II,
                \JJ)\in
                \Obj(\mathfrak{D}_r)}}}\hspace{.8mm}
                \eta\big(F[\II]\times
F[\JJ]\big), F[\O_{k+1}], \dots,
 F[\O_m]\Bigg)\\
=
B_\Eta\big(F[\O_1],\dots,F[\O_m]\big)
\kern10.5cm\\
-
\bigcup_{k=1}^m\
\underset{\II \neq
\pmb{\emptyset} \neq \JJ}
                {\underset{\II\amalg
                \JJ=\O_k}
                {\bigcup\limits_{(\II,
                \JJ)\in
                \Obj(\mathfrak{D}_r)}}}\hspace{.8mm}
B_\Eta\big(
F[\O_1],\dots, F[\O_{k-1}], 
\eta\big(F[\II]\times
F[\JJ]\big), F[\O_{k+1}], \dots,
 F[\O_m]\big).
\end{multline*}
The same argument shows that $\bar 
B_\Eta\big(F_\Eta[\O_1],\dots,F_\Eta[\O_m]\big)$ equals the last
expression where every occurrence of $B_\Eta$ is replaced by $\bar
B_\Eta$. By Lemma~\ref{lem:assoc} ($m$-permutability for $(F,\Eta)$), 
we have
$$B_\Eta\big(F[\O_1],\dots,F[\O_m]\big)
=\bar B_\Eta\big(F[\O_1],\dots,F[\O_m]\big)$$
and
\begin{multline*}
B_\Eta\big(
F[\O_1],\dots, F[\O_{k-1}], 
\eta\big(F[\II]\times
F[\JJ]\big), F[\O_{k+1}], \dots,
 F[\O_m]\big)\\
=\bar B_\Eta\big(
F[\O_1],\dots, F[\O_{k-1}], 
\eta\big(F[\II]\times
F[\JJ]\big), F[\O_{k+1}], \dots,
 F[\O_m]\big),
\end{multline*}
hence
$$
B_\Eta\big(F_\Eta[\O_1],\dots,F_\Eta[\O_m]\big)
=\bar B_\Eta\big(F_\Eta[\O_1],\dots,F_\Eta[\O_m]\big),$$
which establishes our claim.
\end{proof}

\begin{lemma}[\sc Functoriality of $F_\Eta^{(k)}$] 
\label{lem:kfunct}
For each
morphism $\ff: \O\rightarrow\widetilde\O$ in ${\bf Set}^r$ and
every integer $k\geq0$, we have
\[
F[\ff](F_\Eta^{(k)}[\O]) = F_\Eta^{(k)}[\widetilde\O].
\]
\end{lemma}

\begin{proof}
We use induction on $k$, our claim being obvious for $k=0$.
Suppose that the assertion holds for $0\leq k<K$ with some
$K\geq1$. Then, using the definition of $F_\Eta^{(k)}$,
the functoriality of $F_\Eta$ already
demonstrated in Lemma~\ref{lem:funct}, 
the inductive hypothesis, as well as the naturality
of $\Eta$, we find that
\begin{align*}
F[\ff](F_\Eta^{(K)}[\O]) &= F[\ff]
\Bigg(\underset{\O^\prime\subseteq\O}{\bigcup_{\O^\prime\in\Obj({\bf
Set}^r)}} \eta\big(F_\Eta[\O^\prime] \times F_\Eta^{(K-1)}[\O
- \O^\prime]\big)\Bigg)\\[2mm] &=
\underset{\O^\prime\subseteq\O}{\bigcup_{\O^\prime\in\Obj({\bf
Set}^r)}}
\big(F\circ\amalg\big)\big[(\ff\vert_{\O^\prime},
\ff\vert_{\O - \O^\prime})\big]
\big(\eta\big(F_\Eta[\O^\prime] \times
 F_\Eta^{(K-1)}[\O - \O^\prime]\big)\big)\\[2mm]
 &=
\underset{\O^\prime\subseteq\O}{\bigcup_{\O^\prime\in\Obj({\bf
Set}^r)}} 
 \eta\big(F[\ff\vert_{\O^\prime}](F_\Eta[\O^\prime])\times
 F[\ff\vert_{\O-\O^\prime}]
(F_\Eta^{(K-1)}[\O-\O^\prime])\big)\\[2mm] 
&= \underset{\widetilde\O'\subseteq\widetilde\O}{\bigcup_{\widetilde\O'\in\Obj({\bf Set}^r)}}
\eta\big(F_\Eta[\widetilde\O'] \times F_\Eta^{(K-1)}[\widetilde\O -
\widetilde\O']\big)\\[2mm] &= F_\Eta^{(K)}[\widetilde\O].
\end{align*}
\end{proof}

\begin{lemma} \label{lem:Feta1}
The functors $F_\Eta^{(1)}$ and $F_\Eta$ coincide.
\end{lemma}

\begin{proof}
It suffices to show that $F_\Eta^{(1)}[\O] = F_\Eta[\O]$ for every
$\O\in\Obj({\bf Set}^r)$. By (\ref{Eq:FetakBoundary}), this
holds if $\O=\pmb{\emptyset}$, so assume that
$\O\neq\pmb{\emptyset}$. Then, using the definition of $F_\Eta^{(k)}$,
the injectivity of $\Eta$-maps,
Lemma~\ref{lem:assoc3} (3-associativity), 
and Lemma~\ref{lem:emptyset}, we have 
\begin{align*}
F_\Eta^{(1)}[\O] &=
\underset{\O_1\subseteq\O}{\bigcup_{\O_1\in\Obj({\bf Set}^r)}}
\eta\big(F_\Eta[\O_1] \times F_\Eta^{(0)}[\O - \O_1]\big)\\[2mm]
&= \eta\big(F_\Eta[\O] \times
F[\pmb{\emptyset}]\big)\\[2mm] &=
\eta\Bigg(\Bigg(F[\O]\hspace{.8mm}-\hspace{.8mm}
\underset{\II\neq\pmb{\emptyset}\neq\JJ}
{\underset{\II\amalg\JJ=\O}
{\bigcup_{(\II,\JJ)\in\Obj(\mathfrak{D}_r)}}}
\eta\big(F[\II] \times F[\JJ]\big)\Bigg)
\times F[\pmb{\emptyset}]\Bigg)\\[2mm] &= \eta\big(F[\O]
\times F[\pmb{\emptyset}]\big)\hspace{.8mm}-\hspace{.8mm}
\underset{\II\neq\pmb{\emptyset}\neq\JJ}
{\underset{\II\amalg\JJ=\O}
{\bigcup_{(\II,\JJ)\in\Obj(\mathfrak{D}_r)}}}
\eta\Big(\eta\big(F[\II] \times F[\JJ]\big)
\times F[\pmb{\emptyset}]\Big)\\[2mm] &=
F[\O]\hspace{.8mm}-\hspace{.8mm}
\underset{\II\neq\pmb{\emptyset}\neq\JJ}
{\underset{\II\amalg\JJ=\O}
{\bigcup_{(\II,\JJ)\in\Obj(\mathfrak{D}_r)}}}
\eta\big(F[\II] \times F[\JJ]\big)\\[2mm] &=
F_\Eta[\O],
\end{align*}
proving our claim.
\end{proof}

In the next lemma, we require again the concept of a base point, which
was introduced just before Lemma~\ref{lem:Fzerl}.

\begin{lemma} \label{lem:Fkzerl}
For every non-empty $\O\in\Obj({\bf Set}^r)$, each choice
of base point $(\omega,\rho)\in\O$, and every integer $k\geq1$, we
have
\begin{equation}
\label{Eq:FEtaKDec} 
F_\Eta^{(k)}[\O] =
\underset{(\omega,\rho)\in\O_1\subseteq\O} {\coprod_{\O_1\in\Obj({\bf
Set}^r)}} \eta\big(F_\Eta[\O_1] \times
F_\Eta^{(k-1)}[\O-\O_1]\big).
\end{equation}
\end{lemma}

\begin{proof}
The fact that the terms on the right-hand side of
(\ref{Eq:FEtaKDec}) are pairwise disjoint follows from Lemma~\ref{lem:Fdisj},
since a term $\eta\big(F_\Eta[\O_1] \times
F_\Eta^{(k-1)}[\O-\O_1]\big)$ is contained in the larger set
$\eta\big(F_\Eta[\O_1] \times F[\O-\O_1]\big)$. 

An immediate induction using the definition of $F_\Eta^{(m)}$ shows that,
for all $m\ge2$, we have
\begin{equation} \label{eq:Fm}
F_\Eta^{(m)}[\O]=
\underset{\O_1\amalg\cdots\amalg\O_m=\O}{\bigcup_{\O_1,\dots,\O_m\in\Obj({\bf
Set}^r)}}\hspace{.8mm}\eta\big(F_\Eta[\O_1]\times\eta\big(F_\Eta[\O_2]
\times\cdots\times\eta\big(F_\Eta[\O_{m-1}]\times
F_\Eta[\O_m]\big)\cdots\big).
\end{equation}

We first consider the case where $k=1$.
Here, by definition of $F_\Eta^{(0)}[\pmb\emptyset]$, the only
contribution to the union on the right-hand side of
\eqref{Eq:FEtaKDec} arises for $\O_1=\O$. In that situation, we have
\begin{equation*} 
 \eta\big(F_\Eta[\O] \times
F_\Eta^{(0)}[\pmb\emptyset]\big)
=\eta\big(F_\Eta[\O] \times
F[\pmb\emptyset]\big).
\end{equation*}
However, this is also the only contribution on the right-hand side of
the definition of $F_\Eta^{(1)}$ given in \eqref{eq:Fetakdef}, thus
proving \eqref{Eq:FEtaKDec} for $k=1$.

Now we consider the case where $k\ge2$.
If $k\ge3$, then, given $\O\in\Obj(\mathbf{Set}^r)$,
we substitute the right-hand side
of \eqref{eq:Fm} with $m=k-1$ in \eqref{Eq:FEtaKDec}. 
As a result, we obtain
\begin{equation} \label{eq:Feta-basepoint}
\underset{(\omega,\rho)\in\O_1}
{\underset{\O_1\amalg\O_2\amalg\cdots\amalg\O_k=\O}
{\bigcup_{\O_1,\O_2,\dots,\O_k\in\Obj({\bf
Set}^r)}}}
\eta\big(F_\Eta[\O_1]\times\eta\big(F_\Eta[\O_2]
\times\cdots\times\eta\big(F_\Eta[\O_{k-1}]\times
F_\Eta[\O_k]\big)\cdots\big)
\end{equation}
for the right-hand side of \eqref{Eq:FEtaKDec}.
We note that Expression~\eqref{eq:Feta-basepoint} also agrees with
the right-hand side of \eqref{Eq:FEtaKDec} for $k=2$ (taking into
account the fact
that we already know that the union on the right-hand side of
\eqref{Eq:FEtaKDec} is a disjoint union).

Expression~\eqref{eq:Feta-basepoint}
is almost \eqref{eq:Fm} with $m=k$, except that $\O_1$ is
distinguished by having to contain the given base point
$(\omega,\rho)$. However, by Lemma~\ref{lem:Feta-assoc}
($m$-permutability for $(F_\Eta,\Eta)$), the ordering
of $\O_1,\O_2,\dots,\O_k$ in the $\Eta$-bracketing in the union on the
right-hand side of \eqref{eq:Feta-basepoint} is of no relevance. Thus,
the restriction that $(\omega,\rho)\in\O_1$ can be dropped. This 
shows that the right-hand side of \eqref{Eq:FEtaKDec} equals
$F_\Eta^{(k)}[\O]$, as claimed.
\end{proof}

\section{Proofs of Propositions~\ref{lem:F->Fk} and \ref{prop:Disjoint},
and of Theorem~\ref{Thm:MainThm2}}
\label{sec:main2}

We begin this section with the proof of Proposition~\ref{lem:F->Fk}.
With
Lemma~\ref{lem:Fkzerl} in hand, we are finally 
in the position to also establish Proposition~\ref{prop:Disjoint}. 
Theorem~\ref{Thm:MainThm2} is then a
simple consequence of an identity on which the proof of
Proposition~\ref{prop:Disjoint} rests (see \eqref{eq:GF-Fetak} below).
Although the property expressed in this proposition is of a
structural nature, our proof relies in fact on a counting argument.
It would be desirable to find an alternative approach more in keeping
with the actual nature of Proposition~\ref{prop:Disjoint}.

\begin{proof}[Proof of Proposition~{\em\ref{lem:F->Fk}}]
We use induction on $\|\O\|$, where $\|\,.\,\|$ has been defined in
\eqref{eq:omabs}. By (\ref{Eq:FetakBoundary}) and the
definition of $F_\Eta^{(0)}$, the statement holds if $\|\O\|=0$,
that is, if $\O=\pmb{\emptyset}$. Let $\O$ be such that
$\|\O\|=N$ for some integer $N>0$, and suppose that
(\ref{Eq:FOmegaFilt}) holds for all $\O'\in\Obj(\mathbf{Set}^r)$ 
of norm strictly less
than $N$. Then we have $\O\neq\pmb{\emptyset}$, and
therefore
\begin{align*}
\bigcup_{k\geq0} F_\Eta^{(k)}[\O] &= \bigcup_{k\geq1}
F_\Eta^{(k)}[\O]\\[2mm] &=
\bigcup_{k\geq1}\hspace{.8mm}
\underset{\O_1\subseteq\O}{\bigcup_{\O_1\in\Obj({\bf
Set}^r)}} \eta\big(F_\Eta[\O_1] \times
F_\Eta^{(k-1)}[\O-\O_1]\big)\\[2mm] &=
\underset{\O_1\subseteq\O}{\bigcup_{\O_1\in\Obj({\bf Set}^r)}}
\eta\bigg(F_\Eta[\O_1] \times \Big(\bigcup_{k\geq1}
F_\Eta^{(k-1)}[\O-\O_1]\Big)\bigg)\\[2mm] &=
\underset{\pmb{\emptyset}\neq\O_1\subseteq\O}{\bigcup_{\O_1\in\Obj({\bf
Set}^r)}} \eta\bigg(F_\Eta[\O_1] \times \Big(\bigcup_{k\geq0}
F_\Eta^{(k)}[\O-\O_1]\Big)\bigg)\\[2mm] &=
\underset{\O_1\subseteq\O}{\bigcup_{\O_1\in\Obj({\bf Set}^r)}}
\eta\big(F_\Eta[\O_1] \times F[\O-\O_1]\big)\\[2mm] &= F[\O].
\end{align*}
Here, we have used Lemma~\ref{lem:Fzerl} for the last equality, and the inductive
hypothesis in the second but last step (here it is important that
$\O_1\ne\pmb\emptyset$ in order to guarantee that $\|\O-\O_1\|<\|\O\|$).
\end{proof}

\begin{proof}[Proof of Proposition~{\em\ref{prop:Disjoint}} and 
of Theorem~{\em\ref{Thm:MainThm2}}]
For $k\ge0$,
let us define the generating function for $F_\Eta^{(k)}$ by
$$
\GF_{F_\Eta^{(k)}}(z_1,\ldots,z_r) :=
\sum_{n_1,\ldots,n_r\geq0} \sum_{x\in F_\Eta^{(k)}[([n_1],\dots,[n_r])]}
w_{([n_1],\dots,[n_r])}(x)\,
\frac {z_1^{n_1} \cdots
z_r^{n_r}} {n_1!\cdots n_r!}.
$$
Again, in the sequel, we shall suppress the indices to weights $w$ for better
readability, the indices always being clear from the context.

The first step consists in showing that 
\begin{equation} \label{eq:GF-Fetak} 
\GF_{F_\Eta^{(k)}}(z_1,\ldots,z_r) =\frac {1} {k!}
\big(\GF_{F_\Eta}(z_1,\ldots,z_r) \big)^k.
\end{equation}
By definition of $F_\Eta^{(0)}$, the left-hand side of
\eqref{eq:GF-Fetak} equals $1$, so that \eqref{eq:GF-Fetak}
holds for $k=0$. Therefore, we may in the sequel assume that $k\ge1$. 

We now proceed in a manner similar to the proof of
Theorem~\ref{Thm:MainThm1} given in Section~\ref{sec:main1}.
Here, however, we use Lemma~\ref{lem:Fkzerl} instead of Lemmas~\ref{lem:Fzerl}
and \ref{lem:Fdisj}, and we also need the functoriality of
$F_\Eta^{(m)}$ for $m=0,1,2,\dots$ established in
Lemma~\ref{lem:kfunct}. In this way, we obtain
from \eqref{Eq:FEtaKDec} the identity
\begin{multline}
\label{Eq:psiFEtaEval} 
n_\rho\cdot w\big(F_\Eta^{(k)}[([n_1],\ldots, [n_r])]\big) =
{\sum_{0\leq\mu_i\leq n_i\,(1\leq i\leq
r)}}\,
\mu_\rho \bigg(\prod_{1\leq i\leq r}
\binom{n_i}{\mu_i}\bigg) 
w\big(F_\Eta[([\mu_1],\ldots,
[\mu_r])]\big)\\
\kern7cm
\cdot w\big(F_\Eta^{(k-1)}
[([n_1-\mu_1] ,\ldots,
[n_r-\mu_r])]\big).
\end{multline}
Fixing $\rho\in[r]$, multiplying both sides of
(\ref{Eq:psiFEtaEval}) by
\[
 {z_1^{n_1} \cdots z_{\rho-1}^{n_{\rho-1}}
z_\rho^{n_\rho-1}z_{\rho+1}^{n_{\rho+1}}
\cdots z_r^{n_r}} /
{(n_1! \cdots n_r!)},
\]
and summing over all tuples $(n_1, \ldots,n_r)\in\N_0^r$, gives
\begin{multline}
\label{Eq:PsiFEtaEval1}
{\sum_{n_1,\ldots,n_r\geq0}}\,
w\big(F_\Eta^{(k)}[([n_1],\ldots,[n_r])]\big)\,\frac { z_1^{n_1} \cdots
 z_{\rho-1}^{n_{\rho-1}}
z_\rho^{n_\rho-1}z_{\rho+1}^{n_{\rho+1}}
\cdots z_r^{n_r}} {n_1! \cdots n_{\rho-1}!\,
(n_\rho-1)!\,n_{\rho+1}!
\cdots n_r!}\\ 
\kern-4cm =
{\sum_{n_1,\ldots,n_r\geq0}}\
{\sum_{0\leq\mu_i\leq n_i\,(1\leq i\leq
r)}}\,
\frac{w\big(F_\Eta[([\mu_1],\ldots,[\mu_r])]\big)}
{\mu_1! \cdots \mu_{\rho-1}!\,
(\mu_\rho-1)!\, \mu_{\rho+1}!\cdots \mu_r!}\\
\cdot\frac{w\big(F_\Eta^{(k-1)}[([n_1-\mu_1],
\ldots,[n_r-\mu_r])]\big)}{(n_1-\mu_1)! \cdots
(n_r-\mu_r)!}\,
z_1^{n_1} \cdots  z_{\rho-1}^{n_{\rho-1}}
z_\rho^{n_\rho-1}z_{\rho+1}^{n_{\rho+1}}
\cdots z_r^{n_r},
\end{multline}
where, again, $1/(-1)!$ has to be interpreted as $0$.
The left-hand side equals
\[
\frac{\partial \GF_{F_\Eta^{(k)}}}{\partial z_\rho},
\]
while the right-hand side is identified as
\[
\frac{\partial \GF_{F_\Eta}}{\partial z_\rho} \GF_{F_\Eta^{(k-1)}},
\]
whence the equations
\begin{equation} \label{eq:Fetadiff} 
\frac{\partial \GF_{F_\Eta^{(k)}}}{\partial z_\rho}
=
\frac{\partial \GF_{F_\Eta}}{\partial z_\rho} \GF_{F_\Eta^{(k-1)}},
\quad \quad 1\le \rho\le r.
\end{equation}
Assuming inductively that 
$$\GF_{F_\Eta^{(k-1)}}=\frac {1} {(k-1)!}
\big(\GF_{F_\Eta}\big)^{k-1},$$
we infer from \eqref{eq:Fetadiff} that
$$
\GF_{F_\Eta^{(k)}}=\frac {1} {k!}\big(\GF_{F_\Eta}\big)^{k}+
C,
$$
where $C$ is independent of $z_1,z_2,\dots,z_r$. 
Making use of the facts that $\GF_{F_\Eta^{(k)}}(0,\ldots,0) =
0$ (since $k\ge1$) and that $\GF_{F_\Eta}(0,\ldots,0) = 0$,
we see that $C=0$, which proves \eqref{eq:GF-Fetak}.

On the other hand, by Theorem~\ref{Thm:MainThm1}, we know that
$$\GF_F(z_1,z_2,\dots,z_r)=\exp\big(\GF_{F_\Eta}(z_1,z_2,\dots,z_r)\big),$$
or, equivalently,
\begin{equation} \label{eq:GF-F-Fk} 
\GF_F(z_1,z_2,\dots,z_r)=
\sum _{k\ge0} ^{}\frac {1} {k!}
\big(\GF_{F_\Eta}(z_1,z_2,\dots,z_r)\big)^{k}.
\end{equation}
If there were a non-empty intersection between $F_\Eta^{(k_1)}[\O]$
and $F_\Eta^{(k_2)}[\O]$, for some $k_1,k_2$ with $k_1<k_2$ and some
$\O\in\Obj(\mathbf{Set}^r)$, then 
Proposition~\ref{lem:F->Fk} would 
contradict \eqref{eq:GF-F-Fk} and \eqref{eq:GF-Fetak}.
This proves the assertion of Proposition~\ref{prop:Disjoint}.

\medskip
The proof of Theorem~{\ref{Thm:MainThm2}} is now easily completed.
By definition of $\widetilde \GF_{F}(z_1,\ldots,z_r,y)$, we have
$$
\widetilde \GF_{F}(z_1,\ldots,z_r,y)=
\sum _{k\ge0} ^{}y^k\GF_{F_\Eta^{(k)}}(z_1,\ldots,z_r).$$
If we now substitute \eqref{eq:GF-Fetak}, then we immediately
obtain \eqref{Eq:MainThm21}. Identity~\eqref{Eq:MainThm22} results
from using Theorem~\ref{Thm:MainThm1} to express 
$\GF_{F_\Eta}(z_1,\ldots,z_r)$ in terms of
$\GF_{F}(z_1,\ldots,z_r)$ and substituting the result in
\eqref{Eq:MainThm21}. 
\end{proof}

\section{Illustrations, I: Three examples}
\label{sec:illust}

We give here three illustrations for the application of our theory.
In the first and second example below, bipartite graphs are considered. 
Example~\ref{ex:1} is, in some sense, ``standard," since it addresses the
case where the composition operator $\Eta$ consists in ``putting
objects together," so that the combinatorial objects in our 
(in this case, $2$-sort) species are sets of indecomposable objects,
a situation which is well covered by classical species
theory. In Example~\ref{ex:2}, however, the composition operator
$\Eta$ is different, ``non-standard," so that classical species theory
does not apply, but our (extension of species) theory does.
On the other hand, we shall see in Section~\ref{sec:commeta} that this
composition operator is pointwise associative and commutative
(for the precise definition see
\eqref{eq:passoc} and \eqref{eq:pcomm}), and that this is equivalent
to the fact that this case is also covered by Menni's theory in
\cite{MennAA}. As a consequence, this family of composition
operators is closely related to the classical operation of ``putting
objects together." (See Theorem~\ref{thm:commeta} for the precise
statement.) Our last example in this section, Example~\ref{ex:3},
presents an example of a composition operator that is neither pointwise
associative nor pointwise commutative, in other words, a composition
operator that is not covered by Menni's theory in \cite{MennAA}.
A particular aspect demonstrated by Examples~\ref{ex:1} and \ref{ex:2}
that we want to highlight is that
composition operators need not be unique.

\begin{example}[\sc Bipartite graphs I] \label{ex:1}
Let the $2$-sort species $F: {\bf Set}^2\rightarrow{\bf Set}$ be
defined by 
$$F[\O]=F[(\Omega^{(1)},\Omega^{(2)})]:=
2^{\Omega^{(1)}\times \Omega^{(2)}},\quad 
\O=(\Omega^{(1)},\Omega^{(2)})\in\Obj(\mathbf{Set}^2).$$
Thus, $F[(\Omega^{(1)},\Omega^{(2)})]$ can be considered as set of all
bipartite graphs, where the set of ``white" vertices is
$\Omega^{(1)}$ and the set of ``black" vertices is
$\Omega^{(2)}$. For $(\O_1,\O_2)\in\Obj(\mathfrak D_2)$,
$b_1\in F[\O_1]$ and $b_2\in F[\O_2]$, put
$$\eta_{(\O_1,\O_2)}\big((b_1,b_2)\big):=b_1\amalg b_2.$$
This means that $\eta_{(\O_1,\O_2)}$ merely forms the disjoint union
of the bipartite graphs $b_1$ and $b_2$.
Then it is not difficult to see that $\Eta$ is a natural
transformation satisfying (D1). Moreover, $F_\Eta[\O]$ consists of the
{\it connected\/} bipartite graphs with bipartition
$\O=(\Omega^{(1)},\Omega^{(2)})$.

For a weight, we choose $\Lambda=\mathbb Z[t]$ and
$$w_\O(b):=t^{|b|},\quad b\in F[\O].$$
Again, it is not difficult to see that $\ww$ satisfies 
Axioms~(W0)--(W2); that
is, $\ww$ is a $\Lambda$-weight on $(F,\Eta)$. 

Theorem~\ref{Thm:MainThm1} then says that
\begin{equation} \label{eq:bip1} 
\GF_F(z_1,z_2) =
\exp\big(\GF_{F_\Eta}(z_1,z_2)\big),
\end{equation}
where
$$
\GF_F(z_1,z_2) =
\sum_{n_1,n_2\geq0}\ \sum_{b\in F[([n_1],[n_2])]}\,
t^{|b|}
\frac { z_1^{n_1}z_2^{n_2}} {n_1!\, n_2!}
$$
and
$$
\GF_{F_\Eta}(z_1,z_2) =
\sum_{n_1,n_2\geq0}\ \sum_{b\in F_\Eta[([n_1],[n_2])]}\,
t^{|b|}
\frac {z_1^{n_1}z_2^{n_2}} {n_1!\,n_2!}.
$$
However, by straightforward counting, one sees that
$$
\GF_{F}(z_1,z_2) =
\sum_{n_1,n_2\geq0} (1+t)^{n_1n_2}
\frac { z_1^{n_1}z_2^{n_2}} {n_1!\, n_2!}.
$$
From \eqref{eq:bip1}, it then follows that the generating function for
connected bipartite graphs is given by
$$
\GF_{F_\Eta}(z_1,z_2) =
\log\left(\sum_{n_1,n_2\geq0} (1+t)^{n_1n_2}
\frac { z_1^{n_1}z_2^{n_2}} {n_1!\, n_2!}\right),
$$
while \eqref{Eq:MainThm22} implies that
\begin{align} \notag
\widetilde\GF_{F}(z_1,z_2) &:=
\sum_{n_1,n_2\geq0}\ \sum_{b\in F[([n_1],[n_2])]}\,
t^{|b|}y^{\#(\text{connected components of }b)}
\frac { z_1^{n_1}z_2^{n_2}} {n_1!\, n_2!}\\
\label{eq:bip2} 
&=
\left(\sum_{n_1,n_2\geq0} (1+t)^{n_1n_2}
\frac { z_1^{n_1}z_2^{n_2}} {n_1!\, n_2!}\right)^y.
\end{align}

This example can be considered as a $2$-dimensional analogue of the
example in \cite[Sec.~3]{DM} (with the first of the two composition
operators considered there). The knowledgeable reader will recognise 
\eqref{eq:bip2} as the exponential generating function for the Tutte
polynomials of complete bipartite graphs (cf.\ e.g.\ 
\cite[Eq.~(3.10)]{ScSoAA}).
\end{example}

\begin{example}[\sc Bipartite graphs II] \label{ex:2}
Let $F: {\bf Set}^2\rightarrow{\bf Set}$ be
as in Example~\ref{ex:1}.
Here, for
$(\O_1,\O_2)\in\Obj(\mathfrak D_2)$,
where $\O_i=(\Omega_i^{(1)},\Omega_i^{(2)})$, $i=1,2$,
for $b_1\in F[\O_1]$ and $b_2\in F[\O_2]$, we put
$$\eta'_{(\O_1,\O_2)}\big((b_1,b_2)\big):=b_1\amalg b_2\amalg
\left(\Omega_1^{(1)}\times \Omega_2^{(2)}\right)\amalg
\left(\Omega_1^{(2)}\times \Omega_2^{(1)}\right).$$
The graph $\eta'_{(\O_1,\O_2)}\big((b_1,b_2)\big)$ 
can be considered as 
a kind of bipartite completion of the disjoint union
of $b_1$ and $b_2$.
Again, it is not difficult to see that $\Eta'$ is a natural
transformation satisfying (D1). Moreover, $F_{\Eta'}(\O)$ consists of the
{\it complements\/} of connected bipartite graphs with bipartition
$\O=(\Omega^{(1)},\Omega^{(2)})$, where the complement $b^c$ of a
bipartite graph $b\in F[\O]$
is defined as $b^c:=\big(\Omega^{(1)}\times \Omega^{(2)}\big)-b$.

If we now were to choose the weight of Example~\ref{ex:1}, then Axiom~(W2)
would be violated. Instead, with $\Lambda=\mathbb Z[t]$, we set
$$w'_\O(b):=t^{|\Omega^{(1)}|\cdot |\Omega^{(2)}|-|b|},\quad b\in 
F[\O]=F\big[(\Omega^{(1)},\Omega^{(2)})\big].$$
Then it is not difficult to see that $\ww'$ does satisfy 
Axioms~(W0)--(W2); that
is, $\ww'$ is a $\Lambda$-weight on $(F,\Eta')$. 

Theorem~\ref{Thm:MainThm1} then says that
\begin{equation} \label{eq:bip3} 
\GF_F(z_1,z_2) =
\exp\big(\GF_{F_{\Eta'}}(z_1,z_2)\big),
\end{equation}
where
$$
\GF_F(z_1,z_2) =
\sum_{n_1,n_2\geq0}\ \sum_{b\in F[([n_1],[n_2])]}\,
t^{n_1n_2-|b|}
\frac { z_1^{n_1}z_2^{n_2}} {n_1!\, n_2!}
$$
and
$$
\GF_{F_{\Eta'}}(z_1,z_2) =
\sum_{n_1,n_2\geq0}\ \sum_{b\in F_{\Eta'}([n_1],[n_2])}\,
t^{n_1n_2-|b|}
\frac {z_1^{n_1}z_2^{n_2}} {n_1!\,n_2!}.
$$
Again, by straightforward counting, one sees that
$$
\GF_{F}(z_1,z_2) =
\sum_{n_1,n_2\geq0} (1+t)^{n_1n_2}
\frac { z_1^{n_1}z_2^{n_2}} {n_1!\, n_2!},
$$
and we obtain the formulae
$$
\GF_{F_{\Eta'}}(z_1,z_2) =
\log\left(\sum_{n_1,n_2\geq0} (1+t)^{n_1n_2}
\frac { z_1^{n_1}z_2^{n_2}} {n_1!\, n_2!}\right)
$$
and
\begin{align} \notag
\widetilde\GF_{F}(z_1,z_2) &:=
\sum_{n_1,n_2\geq0}\ \sum_{b\in F[([n_1],[n_2])]}\,
t^{n_1n_2-|b|}
y^{\#(\text{connected components of }b^c)}
\frac { z_1^{n_1}z_2^{n_2}} {n_1!\, n_2!}\\
\label{eq:bip4} 
&=
\left(\sum_{n_1,n_2\geq0} (1+t)^{n_1n_2}
\frac { z_1^{n_1}z_2^{n_2}} {n_1!\, n_2!}\right)^y.
\end{align}

This example can be viewed as a $2$-dimensional analogue of the
example in \cite[Sec.~3]{DM} (with the second of the two composition
operators considered there).

The alert reader will have noticed that the $\Eta'$-maps could have been
alternatively defined by
\begin{equation} \label{eq:eta'c}
\eta'_{(\O_1,\O_2)}\big((b_1,b_2)\big):=\big(b_1^c\amalg b_2^c\big)^c,
\end{equation}
where the complements have to be taken in the appropriate complete
bipartite graphs. This construction will be generalised in
Section~\ref{sec:commeta}. 
\end{example}

\begin{example}[\sc Binary functions] \label{ex:3}
Let the ($1$-sort) species $F: {\bf Set}\rightarrow{\bf Set}$ be
defined by 
$$F[\O]:=
\{0,1\}^{\O},\quad 
\O\in\Obj(\mathbf{Set}).$$
For $(\O_1,\O_2)\in\Obj(\mathfrak D_1)$,
$f_1\in F[\O_1]$, and $f_2\in F[\O_2]$, put
$$\big(\eta_{(\O_1,\O_2)}
\big((f_1,f_2)\big)\big)(\omega):=\begin{cases} 
f_1(\omega),&\text{if }\omega\in\O_1,\\
1-f_2(\omega),&\text{if }\omega\in\O_2.
\end{cases}$$
Then it is easy to see that $\Eta$ is a natural
transformation satisfying (D1). Moreover, 
$$F_\Eta[\O]=\begin{cases} \{0_\O,1_\O\},&\text{if }|\O|=1,\\
\{\},&\text{otherwise,}\end{cases}$$
where $0_\O$ and $1_\O$ are the constant functions
on $\O$ taking the value $0$ and $1$, respectively.
We note that, in contrast to Examples~\ref{ex:1} and \ref{ex:2}, 
the $\Eta$-maps of the present example are 
{\it pointwise non-associative and non-commutative}
(cf.\ Section~\ref{sec:commeta}); to be precise, in general we have
$$
\eta_{(\O_1\amalg\O_2,\O_3)}
\Big(\big(\eta_{(\O_1,\O_2)}\big((f_1,f_2)\big),f_3\big)\Big)\ne
\eta_{(\O_1,\O_2\amalg\O_3)}
\Big(\big(f_1,\eta_{(\O_2,\O_3)}\big((f_2,f_3)\big)\big)\Big)
$$
and
$$
\eta_{(\O_1,\O_2)}\big((f_1,f_2)\big)\ne 
\eta_{(\O_2,\O_1)}\big((f_2,f_1)\big).
$$

For the sake of completeness, we remark that,
choosing the trivial weighting
$$w_\O(f):=1,\quad f\in F[\O],$$
Theorem~\ref{Thm:MainThm1} yields the trivial identity
\begin{equation*} 
\GF_F(z) =
\sum _{n\ge0} ^{}2^n\frac {z^n} {n!}
=\exp\big(\GF_{F_\Eta}(z)\big)=\exp(2z).
\end{equation*}
The construction of this example can also be generalised to produce many 
more (multisort) species with
pointwise non-associative and non-commutative
composition operator, see Theorem~\ref{thm:Etwist} in 
Section~\ref{sec:commeta}.
\end{example}

\section{Illustrations, II: Magic squares}
\label{Sec:CMT1} 
The purpose of this section is to illustrate the increased flexibility
of our present multivariate setting. We show that a number of 
generating function identities for combinatorial matrices 
found scattered throughout the literature can be uniformly explained,
and generalised, in the context of our theory.

By a {\em combinatorial matrix} on
$\O=(\Omega^{(1)},\Omega^{(2)})\in\Obj({\bf Set}^2)$ we shall mean any
map
\[
m: \Omega^{(1)}\times\Omega^{(2)}\rightarrow\N_0.
\]
The pair of sets $\O$ is called the {\em support} of $m$. Let
$m_1, m_2$ be two combinatorial matrices with supports
$\O_1=(\Omega_1^{(1)},\Omega_1^{(2)})$ and
$\O_2=(\Omega_2^{(1)},\Omega_2^{(2)})$, respectively, and suppose
that $\O_1\cap\O_2=\pmb{\emptyset}$. Then we define their
{\em direct sum} $m=m_1\oplus m_2$ to be the combinatorial matrix
with support $\O:=\O_1\amalg\O_2$ given by
\[
m(\omega_1,\omega_2):=\begin{cases}
m_1(\omega_1,\omega_2),&(\omega_1,\omega_2)
\in\Omega_1^{(1)}\times\Omega_1^{(2)},\\
m_2(\omega_1,\omega_2),&(\omega_1,\omega_2)
\in\Omega_2^{(1)}\times\Omega_2^{(2)},\\
0,&\mbox{otherwise}.
\end{cases}
\]
A combinatorial matrix $m$ on $\O$ is termed 
$s$-{\em magic},\footnote{Strictly speaking, the correct term here would be
``$s$-semi-magic," since we do not require diagonals to sum up to $s$
as well, see e.g.\ \cite{BCCGAA}. 
However, we prefer the term ``$s$-magic" for the sake of brevity.} $s$
a positive integer, if
\[
\sum_{\omega_2\in\Omega^{(2)}} m(\omega_1,\omega_2) =
s,\quad\omega_1\in\Omega^{(1)},
\]
and
\[
\sum_{\omega_1\in\Omega^{(1)}} m(\omega_1,\omega_2)
=s,\quad\omega_2\in\Omega^{(2)}.
\]
Computing the sum of entries, we find that an $s$-magic matrix is
necessarily square, $|\Omega^{(1)}|=|\Omega^{(2)}|$.
The enumeration of $s$-magic squares has a long history, going back to
MacMahon \cite[\S404--419]{MacMAA}. A good account of the enumerative theory
of magic squares can be found in 
\cite[Sec.~4.6]{StanAP}, with many pointers to further literature.
For more recent work, see for instance \cite{BCCGAA,LoLYAA}.

For $s\in\N$ and
$\O\in \Obj({\bf Set}^2)$, denote by $F_s(\O)$ the set of all
$s$-magic matrices on $\O$, and by $\bar{F}_s(\O)$ the set of
those $s$-magic matrices on $\O$ which do not contain $s$ as an
entry. We thus have mappings
\[
F_s, \bar{F}_s: \Obj({\bf Set}^2) \rightarrow \Obj({\bf Set}),
\]
which we turn into functors $F_s, \bar{F}_s: {\bf
Set}^2\rightarrow{\bf Set}$ by assigning to a morphism
\[
\ff = (f_1,f_2): \O\rightarrow\widetilde\O
\]
in ${\bf Set}^2$ the map (denoted $F_s[\ff]$
respectively $\bar{F}_s[\ff]$) sending a combinatorial
matrix $m$ in the respective domain to $m\circ(f_1^{-1}\times
f_2^{-1})$. Moreover, given $s$ and a finite set $\Omega$, let
$F^\ast_s(\Omega)$ be the set of {\em symmetric} $s$-magic
matrices on $\O=(\Omega,\Omega)$;
that is, combinatorial matrices satisfying
\[
m(\omega_1,\omega_2) = m(\omega_2,\omega_1),\quad
(\omega_1,\omega_2)\in\Omega^2;
\]
and denote by $\bar{F}^\ast_s(\Omega)$ the subset of
$F^\ast_s(\Omega)$ consisting of those matrices with no entry
equal to $s$. Just as above, the maps
\[
F_s^\ast, \bar{F}_s^\ast: \Obj({\bf Set})\rightarrow \Obj({\bf Set})
\]
become functors $F_s^\ast, \bar{F}_s^\ast: {\bf
Set}\rightarrow{\bf Set}$ by assigning to a morphism
$f:\Omega\rightarrow\widetilde\Omega$ 
in ${\bf Set}$ the map sending a
combinatorial matrix $m$ in the respective domain to
$m\circ(f^{-1}\times f^{-1})$.

Next, given $s\in\N$, a pair $(\O_1,\O_2)\in\Obj(\mathfrak{D}_2)$,
 and a pair $(\Omega_1,\Omega_2)\in\Obj(\mathfrak{D}_1)$, the direct sum
construction provides us with injective maps
\begin{align*}
(\eta_{s})_{(\O_1,\O_2)}&: F_s(\O_1) \times F_s(\O_2) \rightarrow
F_s(\O_1\amalg \O_2),\\ 
(\bar{\eta}_s)_{(\O_1,\O_2)}&:
\bar{F}_s(\O_1) \times \bar{F}_s(\O_2) \rightarrow
\bar{F}_s(\O_1\amalg \O_2),\\
({\eta}_{s}^\ast)_{(\Omega_1,\Omega_2)}&:
F^\ast_s(\Omega_1) \times F^\ast_s(\Omega_2) \rightarrow
F^\ast_s(\Omega_1\amalg \Omega_2),\\
{(\bar{\eta}}_{s}^\ast)_{(\Omega_1,\Omega_2)}&:
\bar{F}_s^\ast(\Omega_1)\times
\bar{F}_s^\ast(\Omega_2)\rightarrow\bar{F}_s^\ast(\Omega_1\amalg\Omega_2).
\end{align*}
A certain amount of checking is required in order to convince
oneself that these definitions fit into the framework of
Theorems~\ref{Thm:MainThm1} and \ref{Thm:MainThm2}. 
The next lemma states the corresponding result. 
We leave its proof, which essentially amounts to a routine
verification, to the reader.

\begin{lemma}
\label{Lem:CombMat}
\begin{itemize}
\item[(i)] For each $s \in\N,$ the collection of maps
\[
\Eta_{s}=
\big((\eta_{s})_{(\O_1,\O_2)}\big)_{(\O_1,\O_2)\in\Obj(\mathfrak{D}_2)}
\]
is a natural transformation from the functor $F_s\times F_s$ to
the functor $F_s\circ\amalg$. Analogous statements hold for the
functors $\bar{F}_s, F_s^\ast, \bar{F}_s^\ast,$ and  the families
of maps
\begin{align*}
\bar{\Eta}_{s} &=
\big((\bar{\eta}_{s})_{(\O_1,\O_2)}\big)
_{(\O_1,\O_2)\in\Obj(\mathfrak{D}_2)},\\
{\Eta}_{s}^\ast &=
\big(({\eta}_{s}^\ast)
_{(\Omega_1,\Omega_2)}\big)_{(\Omega_1,\Omega_2)\in\Obj(\mathfrak{D}_1)},\\
{\bar{\Eta}}_{s}^\ast &=
\big(({\bar{\eta}}_{s}^\ast)
_{(\Omega_1,\Omega_2)}\big)_{(\Omega_1,\Omega_2)\in\Obj(\mathfrak{D}_1)}.
\end{align*}
\item[(ii)] For each $s,$ the pair $(F_s,\Eta_{s})$ satisfies
Axiom~{\em (D1),} an analogous statement holding for each of the
other pairs $(\bar{F}_s,\bar{\Eta}_{s}),$
$(F_s^\ast,{\Eta}_{s}^\ast),$ and
$(\bar{F}^\ast_s,{\bar{\Eta}}_{s}^\ast)$.
\end{itemize}
\end{lemma}

It follows from Lemma~\ref{Lem:CombMat} and
Theorem~\ref{Thm:MainThm2}, that Equations (\ref{Eq:MainThm21}) and
(\ref{Eq:MainThm22}) hold for each of the pairs
$(F_s,\Eta_{s})$, $(\bar{F}_s,\bar{\Eta}_{s})$, $(F_s^\ast,
{\Eta}_{s}^\ast)$, and
$(\bar{F}^\ast,{\bar{\Eta}}_{s}^\ast)$; in particular,
we find that
\begin{align}
\widetilde\GF_{F_s}(z_1,z_2,y) &= \exp\big(y
\GF_{(F_s)_{\Eta_{s}}}(z_1,z_2)\big),\label{Eq:CombMat1}\\
\widetilde\GF_{\bar{F}_s}(z_1,z_2,y) &= \exp\big(y
\GF_{(\bar{F}_s)_{\bar\Eta_{s}}}(z_1,z_2)\big),\label{Eq:CombMat2}\\
\widetilde\GF_{F_s^\ast}(z,y) &= \exp\big(y
\GF_{(F_s^\ast)_{{\Eta}_{s}^\ast}}(z)\big),\label{Eq:CombMat3}\\
\widetilde\GF_{\bar{F}^\ast_s}(z,y) &=
\exp\big(y\GF_{(\bar{F}^\ast_s)_{
{\bar{\Eta}}_{s}^\ast}}(z)\big)\label{Eq:CombMat4}.
\end{align}
Note that in these identities the variable $y$ keeps track of the
number of indecomposable matrices into which the matrices which are
counted by the respective generating functions on the left-hand sides
can be decomposed.
Clearly, the generating functions occurring in \eqref{Eq:CombMat1} and
\eqref{Eq:CombMat2} can be viewed as formal power series in $z_1z_2$
and $y$; that is, $z_1z_2$ could be replaced by a single variable. 
However, we prefer to keep $z_1$ and $z_2$ separate, since this is
more in line with our general theory.

We note certain dependencies among the series
$\widetilde\GF_{F_s}$, $\widetilde\GF_{\bar{F}_s}$,
$\widetilde\GF_{F_s^\ast}$,
$\widetilde\GF_{\bar{F}^\ast_s}$; for
instance, we observe that an indecomposable $s$-magic matrix on
$([n_1],[n_2])$ cannot contain an entry equal to $s$, unless
$n_1=n_2=1$. It follows that
\begin{equation*}
\big|(F_s)_{\Eta_{s}}([n_1],[n_2])\big| = \begin{cases}
1+\big|(\bar{F}_s)_{\bar{\Eta}_{s}}([1],[1])\big|,& n_1=n_2=1,\\[2mm]
\big|(\bar{F}_s)_{\bar{\Eta}_{s}}([n_1],[n_2])\big|,&\mbox{otherwise},
                                      \end{cases}
\end{equation*}
and hence, by Equations~(\ref{Eq:CombMat1}) and
(\ref{Eq:CombMat2}),
\begin{equation}
\label{Eq:CombMat5} \widetilde\GF_{\bar{F}_s}(z_1,z_2,y)
=e^{-z_1z_2y}\,\widetilde\GF_{F_s}(z_1,z_2,y).
\end{equation}
Similarly, we have
\begin{equation}
\label{Eq:CombMat6}
\widetilde\GF_{\bar{F}_s^\ast}(z,y) =
e^{-y(z+z^2/2)} \widetilde\GF_{F_s^\ast}(z,y).
\end{equation}
Indeed, for $n=1,2$, there exist indecomposable symmetric $s$-magic
matrices on $([n],[n])$ containing an entry $s$:
\[
(s)\, \text{ and }\,
\left(\begin{matrix}0&s\\s&0\end{matrix}\right).
\]
Now let $m$ be a symmetric $s$-magic matrix on $([n],[n])$ with
$n\geq3$, and suppose that $m$ contains an entry equal to $s$ in
position $(i,j)$. Then, if $i=j$, we have $m=(s)\oplus m^\prime$,
where $m^\prime$ has support $([n]-\{i\},[n]-\{i\})$. If, on the
other hand, $i\neq j$, then, by symmetry, $m$ also contains $s$ in
position $(j,i)$, and we find that $m$ splits as
\[
m=\left(\begin{matrix}0&s\\s&0\end{matrix}\right)\oplus m^\prime,
\]
where $m^\prime$ has support $([n]-\{i,j\},[n]-\{i,j\})$, and is
non-empty since $n\geq3$. Thus, in both cases, $m$ is in fact
decomposable. Hence,
\begin{equation*}
\big|(F_s^\ast)_{{\Eta}_{s}^\ast}([n])\big| = \begin{cases}
                                                        1 +
\big|(\bar{F}^\ast_s)_{{\bar{\Eta}}_{s}^\ast}([n])\big|,&n=1,2,\\[2mm]
\big|(\bar{F}^\ast_s)_{{\bar{\Eta}}_{s}^\ast}([n])\big|,
&\text{otherwise,}
                                                        \end{cases}
\end{equation*}
and (\ref{Eq:CombMat6}) follows from Equations~(\ref{Eq:CombMat3})
and (\ref{Eq:CombMat4}).

The enumeration can be done exactly if $s=2$. For, according to
Birkhoff's Theorem (cf.\ \cite{Birkhoff} or 
\cite[Corollary~8.40]{Aig}), 
a $2$-magic matrix $m$ is the sum of two permutation
matrices, say $p_1$ and $p_2$. 
If $m$ is indecomposable, then the
pair $\{p_1,p_2\}$ is uniquely determined. Premultiplying by
$p_1^{-1}$, we obtain a situation where $p_1$ is the identity;
indecomposability forces $p_2$ to be the permutation matrix
corresponding to a cyclic permutation. So there are $n!(n-1)!$
choices for $(p_1,p_2)$, and half this many choices for $m$
(assuming, as we may, that $n>1$). Note that this formula gives
half the correct number for $n=1$. So we have
\[\big|(F_2)_{\Eta_{2}}([n_1],[n_2])\big| = \begin{cases}
1, & n_1=n_2=n=1,\\ \frac{n!(n-1)!}{2}, & n_1=n_2=n>1,\\ 0, &
n_1\neq n_2, \end{cases}
\]
that is,
\[
\GF_{(F_2)_{\Eta_{2}}}(z_1,z_2) = \frac{1}{2}z_1z_2 -
\frac{1}{2}\log(1-z_1z_2),
\]
and therefore
\begin{equation}
\label{Eq:CombMat7} 
\widetilde\GF_{F_2}(z_1,z_2,y) =
(1-z_1z_2)^{-y/2}\,e^{z_1z_2y/2}
\end{equation}
by Equation~(\ref{Eq:CombMat1}). Also,
\begin{equation}
\label{Eq:CombMat8} 
\widetilde\GF_{\bar{F}_2}(z_1,z_2,y)
= (1 - z_1z_2)^{-y/2}\,e^{-z_1z_2y/2},
\end{equation}
making use of Equation~(\ref{Eq:CombMat5}) and the last result.
Special cases of Identities~\eqref{Eq:CombMat7} and
\eqref{Eq:CombMat8} appear in \cite[Sections~8.1 and 8.3]{ADG} 
(see also \cite[Eqs.~(23) and
(24) in Example~6.11]{Stan2}).
For $s>2$, enumeration is more difficult; see Stanley's paper
\cite{Stan1} and Comtet \cite[pp.~124--125]{Comtet} for comments
in this direction; also Goulden and Jackson~\cite[Sections~3.4 and
3.5]{GJ} for some variations on this counting
problem.\footnote{Note however, that the formula given in
\cite{Comtet} for $s=3$ is erroneous.}

For symmetric matrices, it is again possible to count the
indecomposables with $s=2$. For $n>2$, such a matrix can be
represented as a graph in which every vertex has degree~$2$; loops
are permitted, but contribute only one to the degree of a vertex.
Indecomposability of the matrix is reflected in connectedness of
the graph. So the graphs we must consider are paths (with a loop
at each end) and cycles; and, for $n>2$, their number is
$\frac {1} {2}(n!+(n-1)!)$. 
Including the cases where $n\leq2$, we obtain
\[
\GF_{(F_2^\ast)_{{\Eta}_{2}^\ast}}(z) = 
\frac {z^2} {4}+\frac{z}{2(1-z)}-\frac {1} {2}\log(1-z),
\]
and, hence
\begin{equation}
\label{Eq:CombMat9} 
\widetilde\GF_{F^*_2}(z,y)
= (1-z)^{-y/2}\,\exp\Big(\frac {yz^2} {4}+\frac{y z}{2(1-z)}\Big),
\end{equation}
as well as
\begin{equation}
\label{Eq:CombMat10}
\widetilde\GF_{\bar{F}_2^\ast}(z,y) =
(1-z)^{-y/2}\,\exp\Big(-\frac {yz^2} {4}-yz+\frac{yz}{2(1-z)}\Big).
\end{equation}
Identity~\eqref{Eq:CombMat9} generalises \cite[Eq.~(6.3)]{Gup},
whereas \eqref{Eq:CombMat10} generalises \cite[Eq.~(6.4)]{Gup}.

\section{Simple commutative monoids in species}
\label{sec:commeta}

In this final section, we address the natural question:
{\it`Can one characterise all possible composition operators in
$r$-sort species?{}'} In particular: 
how far can the
composition operator $\Eta$ of our theory 
differ from the standard operation for the species of sets of 
combinatorial structures given by forming the disjoint union,
of which
Example~\ref{ex:1} in Section~\ref{sec:illust} is a prototypical
example? We do not have an answer in general. However,
while addressing this question, we also clarify the relation
of our work to the theory developed by Menni in \cite{MennAA}.

Already Joyal pointed out in \cite[Sec.~7.1]{Joyal} that ($r$-sort) species
are endowed with the structure of a {\it symmetric monoidal
category}. The main objects in Menni's theory \cite{MennAA} are {\it
simple commutative monoids} in an {\it arbitrary} symmetric monoidal
category. He defines the notion of a {\it decomposition} as
a simple commutative monoid that satisfies a certain pullback
condition (see \cite[Def.~2.1]{MennAA}), and, in the case where the symmetric
monoidal category that we start with is the category of ($r$-sort)
species, he shows that a decomposition is equivalent with a
composition operator as defined in Section~\ref{Sec:MainThm}. 
He defines an exponential principle in this general
setting, and he proves this principle to hold for a
large family of symmetric monoidal categories that includes
$r$-sort species (see \cite[Prop.~1.4, Ex.~3.2, and Ex.~3.5
with $I=1+1+\dots+1$, the summand $1$ appearing $r$ times in the 
sum]{MennAA}). He then moves on
to characterise decompositions in symmetric monoidal categories
(see \cite[Sec.~2.6]{MennAA}). 

We shall next discuss the notions mentioned in the previous two
paragraphs in some more detail, in order to provide a better feeling
of what Menni's theory is about. Subsequently, we shall make Menni's
characterisation of decompositions explicit for the case of 
$r$-sort species, that is, of composition operators in the sense of
Section~\ref{Sec:MainThm} that define the structure of a simple
commutative monoid in $r$-sort species. We shall see that the latter
is equivalent to the condition of the composition operator
being ``pointwise associative" and ``pointwise commutative" 
(see \eqref{eq:passoc} and \eqref{eq:pcomm}). For the sake of being
self-contained, and for including the case of weighted species as
well (weights not being addressed in \cite{MennAA}, which however
could be built in without too much effort), we provide an independent
proof. We close this section by exhibiting a large family of examples
of composition operators
(see Theorem~\ref{thm:Etwist}), generalising Example~\ref{ex:3}, 
that are {\it not\/} pointwise associative or commutative and, thus,
do not define the structure of a commutative monoid. These examples
are therefore not included in Menni's theory \cite{MennAA}, which
shows that our axiomatic set-up is wider than Menni's theory in the
special case of $r$-sort species.

\medskip
Let $\Sp_r$ denote the category of $r$-sort species.
Given two species $F$ and $G$ in $\Obj(\Sp_r)$ and $\O\in\Obj({\bf Set}^r)$, 
we define their product $*$ (a special case of the more general concept of 
Day convolution; cf.~\cite{ImKeAA} and \cite[Sec.~3.1]{MennAC}) by
$$
(F* G)[\O]:=\coprod_{\O_1\amalg \O_2=\O}F[\O_1]\times G[\O_2],
$$
with the obvious morphisms.
In order to discuss monoids in species, we need to introduce the
``unit" species $\mathbf 1$ by
$$\mathbf 1[\O]:=\begin{cases} 
\{1\}&\text {if }\O=\emptyset,\\
\emptyset&\text {if }\O\ne\emptyset.
\end{cases}$$
Here, $1$ denotes a ``canonical" element, and the morphisms are the
obvious ones. Then the triple $(\Sp_r,*,\mathbf 1)$
forms a {\it symmetric monoidal category} (see \cite[Sec.~7.1]{Joyal}). 
We refer the reader to 
\cite{KellAA} (see also \cite{MacLAA}) 
for the precise definition of a symmetric monoidal category. 
For our purposes, it suffices to say
that, in the case of $r$-sort species, 
this involves the natural transformation(s)
$$
\alpha=\alpha_{F,G,H}:F* (G* H)\to (F* G)* H
$$
given by
$$
\alpha:(x,(y,z))\to ((x,y),z),\quad \text {for }
x\in F[\O_1],\
y\in G[\O_2],\
z\in H[\O_3],\
$$
where $F,G,H\in\Obj(\Sp_r)$, and where
$\O_1,\O_2,\O_3\in \Obj({\bf Set}^r)$ are pairwise disjoint,
the natural transformation(s)
$$
\beta=\beta_{F,G}:F* G\to G* F
$$
given by
$$
\beta:(x,y)\to (y,x),\quad \text {for }
x\in F[\O_1],\
y\in G[\O_2],\
$$
the natural transformation(s)
$$
\lambda=\lambda_{F}:\mathbf 1* F\to F
$$
given by
$$
\lambda:(1,x)\to x,\quad \text {for }
x\in F[\O_1],\
$$
and the natural transformation(s)
$$
\rho=\rho_{F}:F* \mathbf 1\to F
$$
given by
$$
\rho=\rho_F:(x,1)\to x,\quad \text {for }
x\in F[\O_1],\
$$
so that --- roughly speaking --- all association and commutation laws
that one may think of are satisfied.

A {\it monoid\/} in the symmetric monoidal category $\Sp_r$ is
by definition a triple $(F,*,\mathbf 1)$, where
$F\in\Obj(\Sp_r)$, such that there are natural transformations
$\mu:F* F\to F$ and $\nu:\mathbf 1\to F$ such that the diagrams
\begin{equation} \label{diag:1}
\begin{diagram}
F*(F* F)&\rTo^{\hbox{$\kern10pt \alpha\kern10pt $}}&
(F* F)* F&\rTo^{\mu*\id}&F* F\\
\dTo^{\id* \mu}& & & &\dTo_{\mu}\\
F* F& &\rTo^{\mu}& &F
\end{diagram}
\end{equation}
and
\begin{equation} \label{diag:2}
\begin{diagram}
\mathbf 1* F&\rTo^{\nu*\id}&
F* F&\rTo^{\id*\nu}&F* \mathbf 1\\
&\rdTo _\lambda& \dTo^{\mu}  &\ldTo >{\rho}\\
& &F& &
\end{diagram}
\end{equation}
commute. A monoid $(F,*,\mathbf 1)$ is {\it commutative} if
the diagram
\begin{equation} \label{diag:3}
\begin{diagram}
F* F&&\rTo^{\beta}&&
F* F\\
&\rdTo _\mu&  &\ldTo >{\mu}\\
& &F& &
\end{diagram}
\end{equation} 
commutes. It is called {\it simple} if the natural transformation
$\nu:\mathbf 1\to F$ is unique. 

As Menni explains in \cite[Example~3.2]{MennAA}, a natural
transformation $\Eta$ from $F\times F$ to $F\circ \amalg$ as in
Section~\ref{Sec:MainThm} induces a 
natural transformation $\mu:F*F\to F$ by
$$
F[\O_1]\times F[\O_2] \overset {\eta_{(\O_1,\O_2)}}
\longrightarrow F[\O_1\amalg \O_2]
\overset\id\longrightarrow
F[\O]
$$
(where, as usual, $\O=\O_1\amalg\O_2$ is a partition of $\O$)
and by the universality property of coproducts. Conversely, a 
natural transformation
$\mu:F*F\to F$ induces a natural transformation $\Eta$ from $F\times F$
to $F\circ\amalg$ by
$$
F[\O_1]\times F[\O_2] \longrightarrow (F*F)[\O_1\amalg \O_2]
\overset\mu\longrightarrow
F[\O_1\amalg \O_2]
$$
In this precise sense, natural transformations 
$\mu:F*F\to F$ and natural transformations
$\Eta:F\times F\to F\circ \amalg$ are equivalent notions. 
Under this equivalence, commutativity of the diagrams~\eqref{diag:1}
and \eqref{diag:3} is equivalent to
the composition operator $\Eta$ being {\it pointwise associative} and
{\it pointwise commutative}, respectively. Here,
we say that $\Eta$ is pointwise associative
if, for all pairwise disjoint 
$\O_1,\O_2,\O_3\in\Obj(\mathbf{Set}^r)$, and all elements
$x_1\in F[\O_1]$, $x_2\in F[\O_2]$, $x_3\in F[\O_3]$, we have
\begin{equation} \label{eq:passoc}
\eta_{(\O_1\amalg\O_2,\O_3)}
\Big(\big(\eta_{(\O_1,\O_2)}\big((x_1,x_2)\big),x_3\big)\Big)=
\eta_{(\O_1,\O_2\amalg\O_3)}
\Big(\big(x_1,\eta_{(\O_2,\O_3)}\big((x_2,x_3)\big)\big)\Big),
\end{equation}
and we say that $\Eta$ is pointwise commutative
if, for all $(\O_1,\O_2)\in\Obj(\mathfrak D_r)$, and all elements
$x_1\in F[\O_1]$, $x_2\in F[\O_2]$, we have
\begin{equation} \label{eq:pcomm}
\eta_{(\O_1,\O_2)}\big((x_1,x_2)\big)=
\eta_{(\O_2,\O_1)}\big((x_2,x_1)\big).
\end{equation}
The $\Eta$-maps in Examples~\ref{ex:1} and \ref{ex:2} are instances of
pointwise associative and 
commutative composition operators, while the composition operator in
Example~\ref{ex:3} is
neither pointwise associative nor pointwise commutative.
The notion of pointwise commutativity and associativity 
should {\it not\/} be confused with the commutativity and the 
$3$-associativity proved in Lemmas~\ref{lem:comm} and
\ref{lem:assoc3}, respectively, which are (in general) strictly
weaker assertions. In particular, if $(F,*,\mathbf 1)$ is a monoid,
then all ``permutabilities" in Lemmas~\ref{lem:comm},
\ref{lem:assoc3}, \ref{lem:assoc4}, \ref{lem:assoc},
\ref{lem:Feta-assoc} come for free since they hold already on a
functorial level, but, as Theorem~\ref{thm:Etwist} shows, the
converse is not true; that is, these
``permutabilities" do {\it not\/} guarantee that $(F,*,\mathbf 1)$ forms a
monoid.

Finally, it is easy to see that there is a 
unique natural transformation $\nu:\mathbf 1\to F$ if and only if 
$\vert F[\emptyset]\vert =1$, a condition automatically satisfied
by a composition operator (see Lemma~\ref{lem:emptyset}), and it is
also easy to see that the diagram~\eqref{diag:2} always commutes.

To summarise the above discussion: 
the notion of $(F,*,\mathbf 1)$ being a simple commutative monoid
is equivalent to the corresponding composition operator $\Eta$ being
pointwise associative and commutative.

\medskip
Now that we have discussed the precise relationship between the theory in
\cite{MennAA} (specialised to $r$-sort species) and our setting laid
down in Section~\ref{Sec:MainThm}, we want to make 
Menni's characterisation of simple commutative monoids in the case of 
$r$-sort species, that is ---
in our language --- of pointwise and associative composition
operators, explicit. In order to do so,
we need two preparatory results.
Recall that a {\it species isomorphism} between two $r$-sort species
$F_1$ and $F_2$ is a collection of maps
$\boldsymbol\ph=(\ph_\O)_{{\O\in\Obj(\mathbf{Set}^r)}}$,
where, for each $\O\in\Obj(\mathbf{Set}^r)$,
$$\ph_\O:F_1[\O]\to F_2[\O]$$
is a bijection, with the property that, for every morphism 
$\ff:\O\to\widetilde\O$ in the category
$\mathbf{Set}^r$, the diagram
\begin{equation}
\label{eq:phinat}
\begin{CD}
F_1[\O] @>F_1[\ff]>>F_1[\widetilde\O]\\
@V{\ph_\O}VV
@VV{\ph_{\widetilde\O}}V\\
F_2[\O] 
@>>{F_2[\ff]}>
F_2[\widetilde\O] 
\end{CD}
\end{equation}
commutes. If $F_1$ carries a weak $\Lambda_1$-weight $\ww_1$ and
$F_2$ carries a weak $\Lambda_2$-weight $\ww_2$, where, by ``weak,"
we mean that $\ww_1$ and $\ww_2$ satisfy
Axioms~(W0) and (W1),
but not necessarily (W2) (cf.\ Section~\ref{Sec:MainThm}), 
then an isomorphism
$\boldsymbol\ph:F_1\to F_2$ is called {\it weight-preserving}, if
there exists a ring homomorphism $\lambda:\Lambda_1\to\Lambda_2$ 
such that the diagram 
\begin{equation}
\label{eq:phiweight}
\begin{CD}
F_1[\O] @>(w_1)_{\O}>>\Lambda_1\\
@V{\ph_\O}VV
@VV{\lambda}V\\
F_2[\O] 
@>>{(w_2)_{\O}}>
\Lambda_2
\end{CD}
\end{equation}
commutes.

The lemma below tells us that, if $F_1$ and $F_2$ are two isomorphic $r$-sort
species, where $F_1$ is decomposable with composition operator
$\Eta_1$, then $\Eta_1$ can be lifted to a composition operator for
$F_2$, demonstrating that $F_2$ is decomposable as well.

\begin{lemma} \label{lem:etatrans}
Let $F_1$ and $F_2$ be two isomorphic $r$-sort species, 
where $F_1$ is decomposable
with composition operator $\Eta_1$. Furthermore, let $\boldsymbol\ph$
be an isomorphism between $F_1$ and $F_2$.
Then $F_2$ is decomposable, and 
the family of maps $\Eta_2=((\eta_2)_{(\O_1,\O_2)})_{
(\O_1,\O_2)\in\Obj(\mathfrak D_r)}$ defined by
\begin{multline*}
(\eta_2)_{(\O_1,\O_2)}(x_1,x_2):=
\ph_{\O_1\amalg \O_2}
\Big((\eta_1)_{(\O_1,\O_2)}
\big((\ph^{-1}_{\O_1}(x_1),\ph^{-1}_{\O_2}(x_2))\big)\Big),
\\
x_1\in F_2[\O_1],\ x_2\in F_2[\O_2],
\end{multline*}
is a composition operator for $F_2$.
\end{lemma}

\begin{proof}
We have to show that $\Eta_2$ is a natural transformation from
$F_2\times F_2$ to $F_2\circ\amalg$, and that the pair $(F_2,\Eta_2)$ 
satisfies Axiom~(D1). The former follows immediately from the corresponding
property for $(F_1,\Eta_1)$ and the naturality condition
\eqref{eq:phinat}. In order to verify (D1), we start with
the left-hand side of \eqref{Eq:DecompAx} for the pair $(F_2,\Eta_2)$,
suppressing the indices of $\eta_1,\eta_2,\ph$ for better readability:
\begin{align*}
\eta_2&\big(F_2[\O_1] \times F_2[\O_2]\big)
\cap \eta_2\big(F_2[\widetilde{\O}_1] \times
F_2[\widetilde{\O}_2]\big)\\
&=\ph\Big(\eta_1
\big(\ph^{-1}(F_2[\O_1])\times\ph^{-1}(F_2[\O_2])\big)\Big)
\cap \ph\Big(\eta_1
\big(\ph^{-1}(F_2[\widetilde\O_1])
\times\ph^{-1}(F_2[\widetilde\O_2])\big)\Big)\\
&=\ph\Big(\eta_1\big(
F_1[\O_1]\times F_1[\O_2]\big)\Big)
\cap \ph\Big(\eta_1\big(
F_1[\widetilde\O_1]
\times F_1[\widetilde\O_2]\big)\Big)\\
&=\ph\Big(\eta_1\big(
F_1[\O_1]\times F_1[\O_2]\big)
\cap \eta_1\big(
F_1[\widetilde\O_1]
\times F_1[\widetilde\O_2]\big)\Big).
\end{align*}
Here we have used the injectivity of $\ph$ to obtain the last line.
Now we substitute the right-hand side of \eqref{Eq:DecompAx} for the
pair $(F_1,\Eta_1)$, to obtain
\begin{multline*}
\eta_2\big(F_2[\O_1] \times F_2[\O_2]\big)
\cap \eta_2\big(F_2[\widetilde{\O}_1] \times
F_2[\widetilde{\O}_2]\big)\\
=\ph\Big(
 \eta_1\big(\eta_1\big(F_1[\O_{11}]
\times F_1[\O_{12}]\big) \times \eta_1\big(F_1[\O_{21}]
\times F_1[\O_{22}]\big)\big)\Big),
\end{multline*}
where $\O_{ij}:= \O_i \cap \widetilde{\O}_j$ for
$i,j\in\{1,2\}$. Using $F_1[\O_{ij}]=\ph^{-1}(F_2[\O_{ij}])$ at each
possible place, and inserting $\text{id}=\ph^{-1}\circ\ph$ at two
places, we arrive at
\begin{align*}
\eta_2\big(F_2[\O_1] \times &F_2[\O_2]\big)
\cap \eta_2\big(F_2[\widetilde{\O}_1] \times
F_2[\widetilde{\O}_2]\big)\\
&=\ph\Big(
 \eta_1\big(\ph^{-1}\big(\ph\big(\eta_1\big(\ph^{-1}(F_2[\O_{11}])
\times \ph^{-1}(F_2[\O_{12}])\big)\big)\big) \\
&\kern3cm
\times 
\ph^{-1}\big(\ph\big(\eta_1\big(\ph^{-1}(F_2[\O_{21}])
\times \ph^{-1}(F_2[\O_{22}])\big)\big)\big)\big)\Big)\\
&= \eta_2\big(\eta_2\big(F_2[\O_{11}]
\times F_2[\O_{12}]\big)
\times 
\eta_2\big(F_2[\O_{21}]
\times F_2[\O_{22}]\big)\big),
\end{align*}
which is exactly \eqref{Eq:DecompAx} for the
pair $(F_2,\Eta_2)$.
\end{proof}

The second preparatory result, Proposition~\ref{lem:E(Feta)}
below, states that, given a decomposable $r$-sort species
$F$ with pointwise associative and
commutative composition operator $\Eta$, $F$ is isomorphic to
$E(F_\Eta)$, where $E(F_\Eta)$ denotes the species of sets of
$F_\Eta$-structures (cf.\ \cite[p.~8]{BLL} for the definition of the
species of sets, $E$, and \cite[p.~41]{BLL} for the definition of
composition of species). In rigorous terms,
for $\O\in\Obj(\mathbf{Set}^r)$, the set $E(F_\Eta)[\O]$ can be defined
by
\begin{multline*}
E(F_\Eta)[\O]:=\Big\{\big\{(x_1,\O_1),\dots,(x_k,\O_k)\big\}:x_i\in 
F_\Eta[\O_i],\ i=1,\dots,k,\\
\text{for some $k\in\mathbb N_0$ and }\O_1\amalg\cdots\amalg \O_k=\O,\
\text{all $\O_i$'s being non-empty}\Big\},
\end{multline*}
with the obvious notion of induced morphisms.
If $F$ carries a weak $\Lambda$-weight $\ww$, then $\ww$ can be
lifted to a weak $\Lambda$-weight of $E(F_\Eta)$ by setting
$$w_\O\Big(\big\{(x_1,\O_1),\dots,(x_k,\O_k)\big\}\Big):=
w_{\O_1}(x_1)\cdots w_{\O_k}(x_k).$$

\begin{proposition} \label{lem:E(Feta)}
Let $F$ be a decomposable weighted $r$-sort species
with composition operator $\Eta$, where $\Eta$ is
pointwise associative and 
commutative. Then there exists a weight-preserving isomorphism
between $F$ and $E(F_\Eta)$.
\end{proposition}

\begin{proof}
The starting point is the combination of Lemmas~\ref{lem:Fzerl} and
\ref{lem:Fdisj}. It says that, 
for each non-empty $\O \in \Obj({\bf Set}^r)$ and
every choice of base point $(\omega,\rho)\in\O,$ we have
\begin{equation}
\label{eq:FOmegaDecomp} 
F[\O]=
\underset{(\omega,\rho)\in\O_1\subseteq\O}
{\coprod_{\O_1\in\Obj({\bf Set}^r)}} \eta_{(\O_1,\O-\O_1)}\big(F_\Eta[\O_1]
\times F[\O -\O_1]\big).
\end{equation}

We are now going to construct bijective maps 
$\ps_\O:F[\O]\to E(F_\Eta)[\O]$ by
induction on $\|\O\|$, where $\|\,.\,\|$ has been defined in
\eqref{eq:omabs}. For $\O=\pmb\emptyset$, we have
$|F[\O]|=|E(F_\Eta)[\O]|=1$ by Lemma~\ref{lem:emptyset} respectively
the definition of $E(F_\Eta)$, whence the construction of
$\ps_{\pmb\emptyset}$ is trivial. Henceforth, we shall suppose that
$\|\O\|\ge1$, and we assume that we have constructed maps
$\ps_{\widetilde\O}$ for
all $\widetilde\O\in\Obj(\mathbf{Set}^r)$ with $\|\widetilde\O\|<N$.

Now let $\|\O\|=N$. 
Choose a base point $(\omega,\rho)\in\O,$ and let
$x\in F[\O]$. By \eqref{eq:FOmegaDecomp}, there is a unique $\O_1$
such that $x\in \eta_{(\O_1,\O-\O_1)}\big(F_\Eta[\O_1]
\times F[\O -\O_1]\big)$ and $(\omega,\rho)\in\O_1$. 
Let $(y_1,x_1)$ be the uniquely determined
pair with $y_1\in F_\Eta[\O_1]$ and $x_1\in F[\O -\O_1]$, such that
\begin{equation} \label{eq:psi1} 
(y_1,x_1):=\eta_{(\O_1,\O-\O_1)}^{-1}(x).
\end{equation}
By the inductive hypothesis, there exist uniquely determined
elements $y_2,\dots,y_k$, for some
$k\in\mathbb N$ and $y_i\in F_\Eta[\O_i]$, $i=2,\dots,k$, with
$\O_2\amalg\cdots\amalg\O_k=\O-\O_1$, such that
\begin{equation} \label{eq:psi2} 
\ps_{\O-\O_1}(x_1)=\big\{(y_2,\O_2),\dots,(y_k,\O_k)\big\}.
\end{equation}
Define
$$\ps_{\O}(x):=\big\{(y_1,\O_1),(y_2,\O_2),\dots,(y_k,\O_k)\big\}.$$
We claim that this yields a well-defined bijection $\ps_{\O}:F[\O]\to
E(F_\Eta)[\O]$. What needs to be checked here
first of all is that different
choices of base points would always lead to the same result. So, let
us suppose, that, by choosing a different base
point, we would have obtained
$$\bar\ps_{\O}(x):=\big\{(\bar y_1,\bar\O_1),(\bar y_2,\bar\O_2),\dots,(\bar
y_l,\bar\O_l)\big\},$$
for some $l$, instead.  
Since we must have
$$\O=\O_1\amalg\O_2\amalg\cdots\O_k=
\bar\O_1\amalg\bar\O_2\amalg\cdots\bar\O_l,$$
there is a $j$ such that $(w,\rho)\in\bar\O_j$.
By our inductive construction via \eqref{eq:psi1} and \eqref{eq:psi2},
we have
$$x\in\eta\Big(F_\Eta[\bar\O_1]\times\eta\big(F_\Eta[\bar\O_2]\times
\cdots\times\big(F_\Eta[\bar\O_{l-1}]\times 
F_\Eta[\bar\O_{l}]\big)\cdots\big)\Big).$$
By Lemma~\ref{lem:Feta-assoc} ($m$-permutability for $(F_\Eta,\Eta)$),
this is equivalent to saying that
\begin{equation} \label{eq:xeta2} 
x\in\eta\Big(F_\Eta[\bar\O_j]\times\eta\big(F_\Eta[\bar\O_{\sigma(2)}]\times
\cdots\times\big(F_\Eta[\bar\O_{\sigma(l-1)}]\times 
F_\Eta[\bar\O_{\sigma(l)}]\big)\cdots\big)\Big),
\end{equation}
where $\sigma(2),\dots,\sigma(l-1),\sigma(l)$ is some permutation of
$\{1,\dots,j-1,j+1,\dots,l\}$. If $\bar\O_j\ne\O_1$, then
\eqref{eq:psi1} and \eqref{eq:xeta2} would contradict the disjointness
in \eqref{eq:FOmegaDecomp}. Hence, we must have $\bar\O_j=\O_1$,
and, by our assumption that $\Eta$ be
pointwise associative and commutative, we even
must have $\bar y_j=y_1$. The inductive hypothesis applied to
$\O-\O_1$ then guarantees that, moreover,
$$
\{(y_2,\O_1),\dots,(y_k,\O_k)\}=
\{(\bar y_1,\O_1),\dots,(\bar y_{j-1},\bar\O_{j-1}),
(\bar y_{j+1},\bar\O_{j+1}),\dots,
(\bar y_l,\O_l)\}.$$
This proves that $\ps_\O$ is indeed well-defined.

The facts that each map $\ps_\O$ is a bijection, and 
that the family $\boldsymbol\ps=(\ps_\O)_{\O\in\Obj(\mathbf{Set}^r)}$
is an isomorphism between $F$ and $E(F_\Eta)$, are not hard to
verify. The fact that $\boldsymbol\ps$ is weight-preserving is obvious
from the definition of $\boldsymbol\ps$ and Axiom~(W2) for
$(F,\Eta,\ww)$.
This completes the proof of the proposition.
\end{proof}

If we combine Lemma~\ref{lem:etatrans} and Proposition~\ref{lem:E(Feta)},
then we can say exactly how a decomposable $r$-sort species $F$
with
pointwise associative and commutative composition operator $\Eta$ arises from 
the composition of the species of sets with the species of
$F_\Eta$-structures (``components"). 
We should point out here that, clearly, 
a natural composition operator for $E(F_\Eta)$ is given by
\begin{equation} \label{eq:natcomp} 
(y_1,y_2)\mapsto y_1\amalg y_2, \quad y_1\in E(F_\Eta)[\O_1],
\ y_2\in E(F_\Eta)[\O_2],\ \O_1\amalg\O_2=\O.
\end{equation}

\begin{theorem} \label{thm:commeta}
Let $F$ be a decomposable $r$-sort species
with composition operator $\Eta$, where $\Eta$ is
pointwise associative and 
commutative. Then, for all $(\O_1,\O_2)\in\Obj(\mathfrak D_r)$, 
the composition operator $\Eta$ can be expressed as follows:
\begin{equation} \label{eq:etapsi}
\eta_{(\O_1,\O_2)}(x_1,x_2)=
\ps^{-1}_{\O_1\amalg \O_2}
\Big(
\ps_{\O_1}(x_1)\amalg\ps_{\O_2}(x_2)\Big),\quad 
x_1\in F[\O_1],\ x_2\in F[\O_2],
\end{equation}
where $\boldsymbol\ps$ is the isomorphism between $F$ and $E(F_\Eta)$
constructed in the proof of Proposition~{\em\ref{lem:E(Feta)}}.
\end{theorem}

\begin{proof}
One combines Lemma~\ref{lem:etatrans} with
Proposition~\ref{lem:E(Feta)}, where the role of the isomorphism 
$\boldsymbol\ph$ in Lemma~\ref{lem:etatrans} is played by the family
of maps $\boldsymbol\ps^{-1}$ constructed in the proof of
Proposition~\ref{lem:E(Feta)}. 
\end{proof}

In summary, all decomposable $r$-sort species with
pointwise associative and commutative
composition operator can be constructed from $E(G)$, for some species
$G$, equipped with the natural composition operator as given in 
\eqref{eq:natcomp} (in the case where $G=F_\Eta$), by applying a lift
in the sense of Lemma~\ref{lem:etatrans} via a species isomorphism.
In the more general setting of \cite{MennAA}, the identification of
the ``indecomposable" objects $G$ (denoted by $\text {\tt L}(F,\zeta)$ there) 
is explained 
in \cite[Sec.~2.6]{MennAA}, while the isomorphism between $F$ and
``composite objects built from $G$" (denoted by $\text {\tt E}G=
\text {\tt E\,L}(F,\zeta)$
there) is constructed in \cite[Lemma~2.8]{MennAA}.

In order to see how Example~\ref{ex:2} in Section~\ref{sec:illust} 
fits into the setting of Theorem~\ref{thm:commeta}, recall that the 
isomorphism between $F$ and $E(F_\Eta)$ in that example can be
defined by mapping the bipartite graph $b\in F[\O]$ to its complement
$b^c$, identifying the connected components (in the classical sense
of graph theory) of $b^c$, and forming the set of complements of
these connected components (restricted to the set of vertices which
a component involves). If this isomorphism is inserted in
\eqref{eq:etapsi}, the result is \eqref{eq:eta'c}.

\medskip
We conclude our paper by pointing out that the construction in
Theorem~\ref{thm:commeta} can be ``twisted" to produce
pointwise non-associative
and non-commutative composition operators as well, thereby 
obtaining a large family of examples that still fit under our theory 
but not under Menni's.

\begin{theorem} \label{thm:Etwist}
Let $G$ be a weighted $r$-sort species, and let $\gg:G\to G$ be a
weight-preserving isomorphism. 
We extend $\gg$ to $E(G)$ by setting
$$
g_{\O\amalg\cdots\amalg\O_k}\Big(\big\{(y_1,\O_1),\dots,(y_k,\O_k)\big\}\Big)
=\big\{(g_{\O_1}(y_1),\O_1),\dots,(g_{\O_k}(y_k),\O_k)\big\}.
$$
Then the family
$\Eta=(\eta_{(\O_1,\O_2)})_{(\O_1,\O_2)\in\Obj(\mathfrak D_r)}$ of
maps defined by
\begin{equation*}
\eta_{(\O_1,\O_2)}(x_1,x_2)=
x_1 \amalg g_{\O_2}(x_2),\quad 
x_1\in E(G)[\O_1],\ x_2\in E(G)[\O_2],
\end{equation*}
where $(\O_1,\O_2)\in\Obj(\mathfrak D_r)$, 
is a composition operator for the weighted species $E(G)$.
\end{theorem}

It is obvious from the definition that the composition operator
$\Eta$ of Theorem~\ref{thm:Etwist} will, in general, be 
neither pointwise associative nor pointwise commutative and, thus,
not fit into the theory in \cite{MennAA}. 
Example~\ref{ex:3} in
Section~\ref{sec:illust} provides a typical example of the above
construction, with $G$ given by
$$G[\O]=\begin{cases} \{0_\O,1_\O\},&\text{if }|\O|=1,\\
\{\},&\text{otherwise,}\end{cases}$$
where $0_\O$ and $1_\O$ are the constant functions on $\O$ taking the value
$0$ and $1$, respectively, and where the isomorphism $\gg$ is 
given by
$g_{\O}(0_\O)=1_\O$ and $g_{\O}(1_\O)=0_\O$ for $|\O|=1$.
However, we expect that there are many composition
operators $\Eta$ not obtainable in this way.

\section*{Acknowledgement}
We thank Mat\'\i as Menni for explaining his work in \cite{MennAA} 
to us, and for his patience in discussing the relationship between
his work and the work of the present paper, leading to a
significantly improved presentation of our results, by putting
them in perspective.


\begin{thebibliography}{99}



\bibitem{Aig}
M. Aigner, \textit{Combinatorial Theory}, Springer--Verlag,
Berlin, 1979.

\bibitem{ADG} H. Anand, V.\,C. Dumir, and H. Gupta,
A combinatorial distribution problem, {\it Duke Math. J.} {\bf 33}
(1966), 757--769.

\bibitem{BCCGAA} M. Beck, M. Cohen, J. Cuomo and P. Gribelyuk,
The number of ``magic" squares, cubes, and hypercubes,
{\it Amer.\ Math.\ Monthly} {\bf 110} (2003), 707--717.

\bibitem{BG} E. Bender and J. Goldman, Enumerative uses of
generating functions, \textit{Indiana Univ. Math. J} \textbf{20}
(1971), 753--764.

\bibitem{Birkhoff} G. Birkhoff, Tres observaciones sobre el algebra
lineal, {\it Univ.\ Nac.\ Tucum\'an Rev.\ Ser.~A} {\bf 5} (1946),
147--150.

\bibitem{BLL} F. Bergeron, G. Labelle, and P. Leroux,
\textit{Combinatorial Species and Tree-Like Structures}, Cambridge
University Press,
Cambridge, 1998.

\bibitem{Comtet} L. Comtet, {\em Advanced Combinatorics}, Reidel,
Dordrecht and Boston, 1974.

\bibitem{LoLYAA}  J.\,A. De Loera, F. Liu and R. Yoshida,
A generating function for all semi-magic squares and the volume of the
Birkhoff polytope, {\it J. Algebraic Combin.} {\bf 30} (2009), 113--139.
 

\bibitem{DM} A. Dress and T.\,W. M\"uller, Decomposable
functors and the exponential principle,
{\it Adv. in Math.} {\bf 129} (1997), 188--221.

\bibitem{EhreAA} C. Ehresmann, {\em Cat\'egories et structures},
Dunod, Paris, 1965.

\bibitem{FlSeAA} P. Flajolet and R. Sedgewick, 
{\it Analytic Combinatorics},
Cambridge University Press, Cambridge, 2009.

\bibitem{FoatAZ} D. Foata,
{\it La s\'erie g\'en\'eratrice exponentielle dans les probl\`emes 
d'\'enum\'eration},
S\'eminaire de math\'ematiques sup\'erieures --- \'et\'e
1971, no.~54, Montr\'eal, Les Presses de l'Universit\'e de Montr\'eal,
1974. 

\bibitem{GJ}
I.\,P. Goulden and D.\,M. Jackson, \textit{Combinatorial
Enumeration}, John Wiley \& Sons, New York, 1983.

\bibitem{Gup} H. Gupta, Enumeration of symmetric matrices, {\it Duke
Math. J.} {\bf 35} (1968), 653--659.

\bibitem{ImKeAA}
G.\,B. Im and G.\,M. Kelly, A universal property of the convolution monoidal 
structure, \textit{J. Pure Appl. Algebra} \textbf{43} (1986), 75--88.

\bibitem{Joyal} A. Joyal, Une th\'eorie combinatoire des s\'eries
formelles, \textit{Adv. in Math.} \textbf{42} (1981), 1--82.

\bibitem{KellAA}
G.\,M. Kelly, \textit{Basic Concepts of Enriched Category Theory},
London Mathematical Society Lecture Note Series, vol.~64, 
Cambridge University Press, Cambridge, New York, 1982.

\bibitem{LaLeAZ} G. Labelle and P. Leroux,
An extension of the exponential formula in enumerative combinatorics,
{\it Electron.\ J. Comb.} {\bf 3}(2) (1996), Article~R12, 14~pp.

\bibitem{MacLAA}
S. MacLane, \textit{Categories for the Working Mathematician},
Graduate Texts in Mathematics, vol.~5, Springer--Verlag, New York, 
Berlin, 1971.

\bibitem{MacMAA} P.\,A. MacMahon, {\em Combinatory Analysis},
vol.~2, Cambridge University Press, 1916; reprinted by Chelsea,
New York, 1960.

\bibitem{MennAA} M. Menni, Symmetric monoidal completions and the
exponential principle among labeled combinatorial structures, 
\textit{Theory Appl.\ Categories} {\bf 11} (2003), 397--419.

\bibitem{MennAC} M. Menni, Combinatorial functional and differential 
equations applied to differential posets,
\textit{Discrete Math.} \textbf{308} (2008), 1864--1888.

\bibitem{MennAB} M. Menni, Algebraic categories whose projections
are explicitly free,
\textit{Theory Appl.\ Categories} {\bf 22} (2009), 509--541.

\bibitem{ScSoAA} A.\,D. Scott and A.\,D. Sokal, 
Some variants of the exponential formula, with application to the 
multivariate Tutte polynomial (alias Potts model),
{\it S\'eminaire Lotharingien Combin.} {\bf 61A} (2009), 
Article~B61Ae, 33~pp.



\bibitem{Stan1} R.\,P. Stanley, Linear homogeneous diophantine
equations and magic labelings of graphs, {\it Duke Math. J.} {\bf
40} (1973), 607--632.

\bibitem{Stan2} R.\,P. Stanley, Generating functions, in: MAA Studies
in Mathematics, vol.~17 (G.-C.~Rota, ed.), Math.\ Assoc.\ Am.,
Washington, 1978, 100--141.

\bibitem{StanAP} R.\,P. Stanley, {\em  Enumerative
Combinatorics}, vol.~1, Wadsworth \& Brooks/Cole, 
Pacific Grove, California, 1986; reprinted by Cambridge University
Press, Cambridge, 1998. 

\bibitem{StanBI} R.\,P.~Stanley, 
{\em Enumerative Combinatorics}, vol.~2,
Cambridge University Press, Cambridge, 1999. 

\bibitem{Wilf} H.\,S. Wilf, \textit{generatingfunctionology},
2nd edition, Academic Press, San Diego, 1994.
\end{thebibliography}
\end{document}